\documentclass{amsart}
\usepackage{ucs}
\usepackage[utf8x]{inputenc}

\usepackage{amsmath,amssymb,amsfonts,amstext,amsthm}
\usepackage{relsize}
\usepackage{url}
\usepackage{graphicx}
\usepackage{tikz}
\input xy
\xyoption{all}
\usepackage[all]{xy}
\usepackage{mathrsfs}
\usepackage{hyperref}
\usepackage{graphicx}
\usepackage{stmaryrd}
\usepackage{comment}

\usepackage{cleveref}
\crefname{subsection}{§}{§§}
\Crefname{subsection}{§}{§§}

\graphicspath{{images/}{../images/}}
\usepackage{subfiles}
\usepackage{tikz}

\newcommand{\nwc}{\newcommand}

\nwc{\aaa}{\mathcal{F}}
\nwc{\aap}{\mathcal{F}_{P}}
\nwc{\al}{\alpha}

\nwc{\C}{\mathbb{C}}
\nwc{\cb}{\overline{C}}
\nwc{\ccc}{\mathfrak{c}}
\nwc{\ch}{\widehat{C}}
\nwc{\cin}{\textbf{(v)}}
\nwc{\cl}{C'}
\nwc{\cp}{\mathcal{C}_{P}}
\nwc{\cpll}{\mathfrak{c}_{P'}}
\nwc{\ct}{\widetilde{C}}

\nwc{\dd}{\mathcal{L}}
\nwc{\ddd}{\mathfrak{d}}
\nwc{\ddl}{\mathcal{L}'}
\nwc{\dlp}{\delta_{P}}
\nwc{\doi}{\textbf{(ii)}}

\nwc{\fl}{\flushleft}
\nwc{\fff}{\mathcal{F}}
\nwc{\ffp}{\mathcal{F}_{P}}
\nwc{\ffq}{\mathcal{F}_{Q}}
\nwc{\ffl}{\mathcal{F}'}

\nwc{\G}{\mathcal{G}}
\nwc{\Ga}{\Gamma}
\nwc{\gtl}{\widetilde{g}}

\nwc{\hra}{\hookrightarrow}
\nwc{\hua}{h^{1}(C,\aaa )}

\nwc{\kk}{{\rm K}}

\nwc{\llb}{\mathcal{L}}

\nwc{\mb}{\mathbb}
\nwc{\mc}{\mathcal}
\nwc{\mm}{\mathfrak{m}}
\nwc{\mmp}{\mathfrak{m}_{P}}
\nwc{\mpd}{\mathfrak{m}_{P}^{2}}

\nwc{\nn}{\mathbb{N}}

\nwc{\ob}{\overline{\mathcal{O}}}
\nwc{\obr}{\mathcal{O}^*}
\nwc{\obp}{\overline{\mathcal{O}}_P}
\nwc{\och}{\mathcal{O}_{\hat{C}}}
\nwc{\oh}{\hat{\mathcal{O}}}
\nwc{\ohp}{\hat{\mathcal{O}}_{P}}
\nwc{\ol}{\mathcal{O}'}
\nwc{\oma}{\Omega (\mathfrak{a})}
\nwc{\omo}{\Omega (\mathcal{O})}
\nwc{\oo}{\mathcal{O}}
\nwc{\op}{\mathcal{O}_P}
\nwc{\opc}{\mathcal{O}_{P,C}}
\nwc{\oph}{\hat{\mathcal{O}}_{P}}
\nwc{\opl}{\mathcal{O}_{P}'}
\nwc{\oplc}{\mathcal{O}_{P,C}'}
\nwc{\opll}{\mathcal{O}_{P'}}
\nwc{\opt}{\tilde{\mathcal{O}}_{P}}
\nwc{\optt}{{\mathcal{O}}_{\tilde{P}}}
\nwc{\oq}{\mathcal{O}_{Q}}
\nwc{\oqt}{\tilde{\mathcal{O}}_{Q}}
\nwc{\ot}{\widetilde{\mathcal{O}}}
\nwc{\overop}{\bar{\oo}_{P}}

\nwc{\pb}{\overline{P}}
\nwc{\pbb}{P^*}
\nwc{\pbi}{\overline{P_{i}}}
\nwc{\pbr}{\overline{P_{r}}}
\nwc{\pgmd}{\mathbb{P}^{g+2}}
\nwc{\pgmu}{\mathbb{P}^{g+1}}
\nwc{\ph}{\hat{P}}
\nwc{\pp}{\mathbb{P}}
\nwc{\prv}{\noindent\textbook{Proof}:}
\nwc{\pt}{\widetilde{P}}
\nwc{\ptl}{\tilde{P}}
\nwc{\pum}{\mathbb{P}^{1}}

\nwc{\qh}{\hat{Q}}
\nwc{\qtl}{\tilde{Q}}
\nwc{\qua}{\textbf{(iv)}}

\nwc{\ra}{\rightarrow}
\nwc{\rh}{\hat{R}}

\nwc{\sei}{\textbf{(vi)}}
\nwc{\sep}{\beq\ast\ \ast\ \ast\enq}
\nwc{\sig}{\sigma}
\nwc{\Sig}{\Sigma}
\nwc{\ssp}{S_{P}}
\nwc{\sss}{{\rm S}}

\nwc{\tre}{\textbf{(iii)}}

\nwc{\um}{\textbf{(i)}}

\nwc{\vpb}{v_{\overline{P}}}
\nwc{\vtxp}{\widetilde{V}_{x,P}}
\nwc{\vxp}{V_{x,P}}

\let \wt=\widetilde
\let \mc=\mathcal

\nwc{\wh}{\hat{\omega}}
\nwc{\whp}{\hat{\omega}_{P}}
\nwc{\woch}{\omega\cdot\mathcal{O}_{\hat{C}}}
\nwc{\woh}{\omega\cdot\hat{\mathcal{O}}}
\nwc{\ww}{\omega}
\nwc{\wwb}{\omega^*}
\nwc{\wwct}{\omega _{\widetilde{C}}}
\nwc{\wwh}{\widehat{\omega}}
\nwc{\wwhp}{\widehat{\omega}_P}
\nwc{\wwp}{\omega _{P}}
\nwc{\wwt}{\widetilde{\omega}}
\nwc{\wwtp}{\widetilde{\omega}_P}

\nwc{\zz}{\mathbb{Z}}

\newtheorem{coro}{Corollary}[section]

\newtheorem{dfn}[coro]{Definition}

\newtheorem{lemma}[coro]{Lemma}

\newtheorem{prop}[coro]{Proposition}

\newtheorem{rem}[coro]{Remark}

\newtheorem{thm}[coro]{Theorem}
\newtheorem{conj}[coro]{Conjecture}

\let \fl=\flushleft

\let \ga=\gamma
\let \sub=\subset
\let \be=\beta
\let \al=\alpha
\let \pr=\prime
\let \la=\lambda

\let \mf=\mathfrak
\let \ov=\overline

\usepackage[margin=3cm]{geometry}
\usepackage{xy}

\title[Arithmetic inflection formulae]{Arithmetic inflection formulae for linear series on hyperelliptic curves}

\title{Arithmetic inflection formulae for linear series on hyperelliptic curves}
\author{Ethan Cotterill}

\address{Instituto de Matem\'atica, UFF, Rua Prof Waldemar de Freitas, S/N,
24.210-201 Niter\'oi RJ, Brazil}

\email{cotterill.ethan@gmail.com}

\author{Ignacio Darago}

\address{Department of Mathematics, University of Chicago, 5734 S. University Avenue, Chicago, IL 60637}

\email{idarago@math.uchicago.edu}

\author{Changho Han}

\address{Department of Mathematics, University of Georgia, Athens, GA 30602}

\email{Changho.Han@uga.edu}

\begin{document}

\begin{abstract}
Over the complex numbers, {\it Pl\"ucker's formula} computes the number of inflection points of a linear series of fixed degree and projective dimension on an algebraic curve of fixed genus. Here we explore the geometric meaning of a natural analogue of Pl\"ucker's formula and its constituent local indices in $\mb{A}^1$-homotopy theory for certain linear series on hyperelliptic curves defined over an arbitrary field.
\end{abstract}

\maketitle                   





\tableofcontents  
\section{Introduction}\label{intro}
\subsection{Inflection of linear series on algebraic curves}
The {\it inflectionary index} of a linear series of rank $r$ and degree $d$ in a point $p$ of an algebraic curve, as measured by the total deviation of the vanishing orders of global sections in $p$ from the generic sequence $(0,1,\dots,r)$, is 
a fundamental local invariant. Over the complex numbers, a classical formula of Pl\"ucker expresses the sum of inflectionary indices as $p$ varies, as an explicit polynomial in $r$, $d$, and the genus $g$ of the underlying curve. Pl\"ucker's inflection formula plays a foundational role in the Eisenbud and Harris' reworking of the {\it Brill--Noether theorem} for general curves of genus $g$; see \cite[Prop. 1.1]{EH}. Pl\"ucker's formula notwithstanding, the behavior of the individual inflectionary indices that contribute to the inflection of a given series are quite mysterious, even over $\mb{C}$; indeed, the stratification of curves in moduli according to the local inflectionary behavior of their canonical series is ongoing \cite{EH2,P}. If we furthermore replace $\mb{C}$ by a non-algebraically closed field, it is natural to wonder whether an analogue of Pl\"ucker's formula is still operative, and how we may interpret its constituent inflectionary indices. The answer to the first question is often ``yes", and $\mb{A}^1$-homotopy theory provides the tools both to produce a conserved global quantity \`a la Pl\"ucker and to interpret its local contributions.

\subsection{\texorpdfstring{The $\mb{A}^1$}{A^1}-homotopy machine}
$\mb{A}^1$-homotopy theory was originally developed by Morel, Voevodsky and others as a purely algebraic approach to homotopy theory modeled on Grothendieck's approach to algebraic geometry. Recently, J. Kass and K. Wickelgren, M. Levine, and M. Wendt have applied $\mb{A}^1$-homotopy theory to spectacular effect to investigate enumerative algebraic geometry over fields $F$ other than $\mb{C}$. In this paper, we will apply the Kass--Wickelgren program, as developed in \cite{KW1,KW2}, to investigate the inflection of linear series on hyperelliptic curves. To this end, we will exhibit inflection as the Euler class of a vector bundle $\mc{E}$, and then show that $\mc{E}$ is (relatively) orientable. These conditions imply that a well-defined {\it $\mb{A}^1$-inflection class} exists in the Grothendieck--Witt group of $F$, which we review below.

\subsection{Inflection as an Euler class}\label{subsec:inf_Euler}
Let $X$ denote a smooth projective curve of genus $g$. A $g^r_d$ on $X$ is a linear series comprised of a degree-$d$ line bundle $L$, together with an $(r+1)$-dimensional subspace of sections $V \subset H^0(X,L)$. By the {\it inflection divisor} of $(L,V)$ on $X$ we mean the nonsurjective locus of the natural evaluation morphism
\begin{equation}\label{evaluation_morphism}
V \otimes \mathcal{O} \ra J^{r+1}(L)
\end{equation}
where $J^{r+1}(L)$ is the principal parts bundle with fibers $H^0(L/L(-(r+1)p))$, $p \in X$. Taking top exterior powers in \eqref{evaluation_morphism}, the inflection divisor is precisely the zero locus of a nonvanishing section $s \in H^0(\mc{E})$, where $\mc{E}:=\det J^{r+1}(L)$. 
Thus the class of this divisor is $e(\mc{E})$, the Euler class (i.e., top Chern class) of $\mc{E}$. A fundamental fact (see Proposition~\ref{prop:jet_exact_seq}) is that for every positive integer $k$, $J^{k+1}(L)$ and $J^k(L)$ are related by an exact sequence, namely
\begin{equation}\label{eq:jet_exact_seq}
0 \ra L \otimes K_X^{\otimes k} \ra J^{k+1}(L) \ra J^k(L) \ra 0
\end{equation}
where $K_X$ is the canonical bundle of $X$; it follows that $\mc{E}=L^{\otimes (r+1)} \otimes K_X^{\otimes \binom{r+1}{2}}$.

\subsection{Relative orientability}\label{subsec:Rel_Orient}
Recall from \cite{KW2} that a vector bundle $\mc{F}$ on an algebraic variety $Y$ is {\it relatively orientable} when $\mc{H} := \mc{H}\mathrm{om}(\det T_Y,\det \mc{F})$ is isomorphic to the tensor square $\mc{L}^{\otimes 2}$ of some line bundle $\mc{L}$ on $Y$. 
When $Y=X$ is a curve and $\mc{F}=\mc{E}$ is as above, the corresponding hom-bundle $\mc{H}$ is isomorphic to $L^{\otimes (r+1)} \otimes K_X^{\otimes (\binom{r+1}{2}+1)}$. In particular, $\mc{H}$ is a line bundle of degree $(r+1)(d+r(g-1))+ 2g-2$. When $X$ is a {\it complex} curve, multiplication by any fixed integer defines an isogeny of the Jacobian of $X$, and
it follows that $\mc{F}$ is relatively orientable if and only if $\phi(d,g,r):=(r+1)(d+r(g-1))$ is an even number. For example, when $X=E$ is of genus 1, every complete linear series of degree $d$ is of rank $r=d-1$, and correspondingly $\phi$ is always even; so $\mc{H}$ is always relatively orientable over $\mb{C}$. Over an elliptic curve $E=(E,0_E)$ defined over a general base field $F$ not of characteristic 2, deciding whether $\mc{F}$ is relatively orientable may in principle be decided by checking that $\deg \mc{H}$ is even and by applying the ``divisibility by 2'' criterion of \cite[Thm. 2.1]{BZ} for a $F$-rational point $P$ of $E$, where $P$ is the unique point for which $\mc{H} \cong \mc{O}((\deg \mc{H}-1) \cdot 0_E) \otimes \mc{O}(P)$.\footnote{Here we identify the elliptic curve $E$ as a group with its Jacobian via $P \mapsto \mc{O}(P-0_E)$, so the equation $P=2Q$ amounts to the statement that $\mc{O}(P-0_E) \cong \mc{O}(Q-0_E)^{\otimes 2}$ as line bundles on $E$.}
On the other hand, for an arbitrary choice of complete linear series of degree $d \geq 2g-1$ on a smooth curve of genus $g$ greater than 1 we have $r=d-g$, and correspondingly $\phi=g(d-g+1)^2$, which is odd when $d$ and $g$ are odd, and even otherwise.

\subsection{Inflection of linear series on hyperelliptic curves}\label{subsec:generalized_torsion}
In this paper, the primary situation of interest will be that in which $X$ is {\it hyperelliptic} and $L$ is the line bundle obtained from any multiple of the $g^1_2$ on $X$. Relative orientability over an arbitrary base field is then trivial, as the hyperelliptic projection $\pi: X \ra \mb{P}^1$ is of degree 2 and both $L$ and $K_X$ are pullbacks of line bundles on $\mb{P}^1$. Note that when $X=E$ is of genus {\it one}, inflection points of any complete linear series of degree $N$ are in bijection with $N$-torsion points on $E$; so the study of inflection on $X$ generalizes the study of torsion on $E$. The analogy between inflection and torsion is itself not new (see, e.g., \cite{Si}), though to our knowledge it hasn't been examined systematically through the lens of $\mb{A}^1$-homotopy theory.

\medskip
Our strategy for calculating inflection on a given hyperelliptic curve is a two-step process, which involves separately computing contributions arising from the hyperelliptic ramification locus and its complement. Roughly speaking, inflection that arises from the ramification locus is constant in moduli, while inflection supported away from the ramification locus varies nontrivially when the underlying (marked) curve moves. This in turn leads naturally to what we call {\it inflectionary varieties} associated to any given flat family of linear series on marked hyperelliptic curves, whose points parameterize those inflection points supported away from the ramification locus. They are analogues of modular curves $X_1(N)$ that parameterize pairs $((E,0),p)$ of elliptic curves $(E,0)$ together with $N$-torsion points $p \in E[N]$.
The study of rational points on $X_1(N)$ is a classical and enduring area of inquiry; here we initiate an experimental study of rational points on inflectionary curves derived from Legendre and Weierstrass pencils of elliptic curves.

\subsection{The Grothendieck--Witt group}\label{subsec:GW}
We will be interested in the class of $e(\mc{E})$ in the Grothendieck--Witt group ${\rm GW}(F)$ of an arbitrary field $F$. Here ${\rm GW}(F)$ is the (additive) group completion of the monoid of symmetric nondegenerate bilinear forms. It is worth recalling here that ${\rm GW}(F)$ has an explicit presentation in terms of generators and relations; see \cite{Lam}. We use $\langle a \rangle$ to denote the class of $a \in F$. The group ${\rm GW}(F)$ contains a distinguished {\it hyperbolic form} $\mb{H}:=\langle 1 \rangle + \langle -1 \rangle$.

\medskip
In the classical situation, we have ${\rm GW}(\mb{C}) \cong \mb{Z}$, which reflects the fact that any quadratic form over the complex numbers is determined up to isomorphism by its rank. 
The nonclassical situations of primary interest to us in the sequel will be $F=\mb{R}$ and $F=\mb{F}_q$, in which $q=p^n$ is a finite prime power. We note that ${\rm GW}(\mb{R}) \cong \mb{Z}^2$ (resp., ${\rm GW}(\mb{F}_q) \cong \mb{Z} \times \mb{F}_q^{\ast}/(\mb{F}_q^{\ast})^2$); indeed, quadratic forms over the real numbers (resp., over a finite field) are characterized by rank and signature (resp., rank and discriminant, modulo squares).

\subsection{Enriched Euler classes and local indices}\label{subsec:loc_glob_Euler}
According to \cite[Thm 1.1]{BW}, the class in ${\rm GW}(F)$ of the Euler class $e(\mc{E})$ may be recovered as a sum 
\[
e(\mc{E})=\sum_{p: s(p)=0}\mathrm{ind}_p(s)
\]
of local Euler indices $\mathrm{ind}_p(s)$ over (all points of) the vanishing locus of the section $s \in H^0(\mc{E})$ described above, provided $s$ has isolated zeroes. In particular, $e(\mc{E})$ is independent of the particular section $s$ chosen. It will turn out that our global Euler classes are themselves uninteresting inasmuch as they are uniform, whereas the local indices reflect the features of our particular choice of base field $F$.

\subsection{Calculating local Euler indices via local coordinates and trivializations}\label{subsec:compute_loc_Euler}
Computing the local Euler indices $\mathrm{ind}_p(s)$ that contribute to $e(\mc{E})$, in turn, is a three-step procedure. We begin by calculating a {\it local Wronskian} expression for the determinant of \eqref{evaluation_morphism} in an \'etale chart of the inflection point $p$ in question. For $\mb{A}^1$-homotopy theory, we need to work in the more refined Nisnevich topology, however, so in a second step we rewrite our local Wronskian in terms of a Nisnevich uniformizer, using standard facts about how Wronskians transform under changes of coordinates. In practice, the \'etale charts arise from projections to the coordinate axes, while the Nisnevich 
charts are associated with generic projections. Finally, we apply a linear algebraic result originally due to Scheja and Storch (we follow \cite{KW3}) to extract $\mathrm{ind}_p(s)$ from our Nisnevich local Wronskian.
The output of this procedure is a {\it trace} of a certain class in ${\rm GW}(k(p))$, where $k(p)$ is the splitting field of $p$ and the trace is canonically induced by the usual field trace of $k(p)$ over $F$.

\medskip
In this paper, we apply this procedure to explicitly compute local Euler indices $\mathrm{ind}_p(s)$ of inflection points $p$ of complete linear series $|2\ell\infty_X|$ on an odd hyperelliptic curve $X$. We assume that the characteristic of the base field $F$ is not 2, which ensures that $X$ admits an affine model of the form $y^2=f(x)$.
We assume, furthermore, that $F$ is perfect and that either (1) $\ell \le g$ and $\ell \equiv 1 \pmod{4}$; or else (2) $\ell > g$ and $2\ell-g \equiv 1 \pmod{4}$, in order to guarantee that the generic projections mentioned above are compatible with our choice of relative orientation in section~\ref{subsec:Rel_Orient} (see Remarks~\ref{rmk:perfect_field} and \ref{rmk:compatibility}). With these assumptions in mind, we break the local Euler index computation based on whether the hyperelliptic projection $\pi: X \rightarrow \mb{P}^1$ is ramified or split over $p$; see Remark~\ref{rmk:infl_poly_cond} for a discussion of the difficulties that arise in the inert case. Our Theorems~\ref{thm:euler_index_ramif_l_leq_g}, \ref{thm:euler_index_ramif_l>g}, \ref{thm:infl_conc_ramif_infty} and \ref{thm:infl_conc_ramif_infty_2} deal with the case in which $\pi$ is ramified over $p$, assuming that determinants of certain Gessel--Viennot-type matrices $M(\ell,g)$ are not divisible by the characteristic of $F$; see Remark~\ref{rem:Gessel--Viennot}. Theorem~\ref{thm:local_Euler_index_infl_poly} deals with the case in which $\pi$ is split over $p$. 


\subsection{Inflection away from the hyperelliptic ramification locus}\label{subsec:infl_general}
To calculate inflectionary indices {\it away} from the hyperelliptic ramification locus $R_{\pi}$, we introduce an additional computational device: the {\it inflection polynomial} $P_{g,\ell}(x)$, whose roots parameterize the $x$-coordinates over $\ov{F}$ of inflection points of $|2\ell \infty_X|$ supported on the complement of $R_{\pi}$. It is defined by a functional equation \eqref{hasse_inflection_poly}, and is a determinant in {\it atomic inflection polynomials} $P_{g,g+1}(x)$ 
that may be generated recursively; see equation~\eqref{eq:reduced_recursion}. In Theorem~\ref{thm:local_Euler_index_infl_poly}, we derive an expression for the $\mb{A}^1$-inflectionary index in a point $(\ga, \pm \sqrt{f(\ga)})$ of $X$, under the assumption that $f(\ga)$ is a quadratic residue in the residue field $k(\ga)$; see Remark~\ref{rmk:infl_poly_cond}. A crucial point is that any flat family of hyperelliptic curves $X_{\la}$ over a $d$-dimensional base defines a $d$-dimensional {\it inflectionary variety} cut out by the inflection polynomial $P_{g,\ell}(x)$, which now varies with the modular parameters $\la$. In this paper, we study the properties of atomic inflectionary {\it curves} derived from Weierstrass pencils of elliptic curves, and in so doing we are led to Conjecture~\ref{atomic_inflectionary_curves}, which makes a precise prediction about the number of $\mb{R}$-rational points of the atomic Weierstrass inflectionary curve as a function of the underlying pencil parameter.

\subsection{Outline}

The paper following this introduction is organized as follows.
In \autoref{sec:blueprint} we detail our strategy for computing arithmetic inflection of arbitrary multiples of the $g^1_2$ on $X$. We begin by reviewing linear series on hyperelliptic curves, including their explicit presentations in terms of coordinates, and their associated inflection divisors. We then review {\it Nisnevich} coordinate charts, as these are essential for calculating local Euler indices. Lemma~\ref{lem:locEuler} is a now-standard general result that characterizes the local Euler index at an isolated zero of arbitrary multiplicity. Theorem~\ref{thm:global_Euler_class} and its Corollary~\ref{thm:global_Euler_class_inflection} together give a Pl\"ucker-type formula for the class of the inflection locus of an arbitrary multiple $\ell g^1_2$ of the $g^1_2$ on a hyperelliptic curve. Subsections~\ref{subsec:toric_geometry} and \ref{subsec:nisnevich_coords_and_projections} are devoted to an in-depth discussion of the particular Nisnevich coordinate charts we use in calculating those local Euler indices that contribute to the global inflection class; we construct these using toric geometry and linear projections. In the concluding subsection~\ref{subsec:numerology} we account for the various numerical hypotheses we make throughout the paper, and their significance.

\medskip
Section~\ref{sec:infl} is the technical core of this paper, in which we compute local Euler multiplicities at points of the ramification locus $R_{\pi}$ of the hyperelliptic projection $\pi$. We use 
{\it Hasse} derivatives, as these behave better than usual derivatives in positive characteristic, to compute local Wronskians. Subsections~\ref{subsec:Hasse_Witt} and \ref{subsec:infl_jet} recapitulate their basic theory, as well as the bundle of principal parts $J^{r+1}(L)$ associated to a $g^r_d$ using these. For lack of a suitable reference, we construct $J^{r+1}(L)$ from scratch using local trivializations, and we prove Proposition~\ref{prop:jet_exact_seq}, which we use in Corollary~\ref{thm:global_Euler_class_inflection}. In subsection~\ref{subsec:Hasse_Wronskian} we then construct the Hasse Wronskians required for our calculation of local Euler indices. Theorems~\ref{thm:infl_conc_ramif_1} and \ref{thm:infl_conc_ramif_infty} (resp., Theorems~\ref{thm:infl_conc_ramif} and \ref{thm:euler_index_ramif_l_leq_g_infty}) give explicit formulae for the Hasse Wronskian in terms of a Nisnevich (resp., \'etale) uniformizer at a point of $R_{\pi}$ whenever $\ell \leq g$ (resp., $\ell>g$), which we subsequently promote in Theorems~\ref{thm:euler_index_ramif_l_leq_g} and \ref{thm:infl_conc_ramif_infty_2} (resp., Theorems~\ref{thm:euler_index_ramif_l>g} and \ref{thm:euler_index_ramif_l>g_infty}) to explicit expressions for the corresponding local Euler indices; when $\ell>g$, ``explicit" means a combinatorial count of nonintersecting lattice paths in the plane. In subsection~\ref{subsec:hasse_inflection_polys}, we generalize (and redefine) the {\it inflection polynomials} studied in \cite{CG} in terms of Hasse derivatives; their roots parameterize the $x$-coordinates of $\overline{F}$-inflection points away from $R_{\pi}$ whenever $\ell>g$, and from them it is straightforward to read off the corresponding local Euler indices. The case $\ell=g+1$ is distinguished in that the corresponding {\it atomic} inflection polynomials are prescribed recursively (Proposition~\ref{prop:recursion}), while general inflection polynomials are determinants in atomic inflection polynomials (Proposition~\ref{inflection_poly_determinantal_presentation}).

\medskip
Finally, in \autoref{sec:geometric_interpretations}, we give concrete interpretations of our formulas for local Euler indices over $\mb{R}$, $\mb{F}_q$, and $\mb{C}(\!(t)\!)$; and we study inflectionary curves (defined by inflectionary polynomials) derived from Lefschetz and Weierstrass pencils of elliptic curves. In Conjecture~\ref{atomic_inflectionary_curves} (resp., Proposition~\ref{atomic_inflectionary_curves_over_Fq}) we make some explicit speculations regarding the $F$-rationality loci of these curves when $F=\mb{R}$ (resp., $F=\mb{F}_q$).

\subsection{Conventions} Hereafter $\mb{N}$ denotes the natural numbers, including zero; while $\ov{\mb{N}}:=\mb{N} \cup \{\infty\}$.

\section{A blueprint for arithmetic inflection on a hyperelliptic curve}\label{sec:blueprint}

\subsection{Linear series on hyperelliptic curves}\label{subsec:lin_series_hypell}
Let $X$ denote a smooth hyperelliptic curve of genus $g$, defined over a base field $F$ of characteristic not equal to 2. By definition, $X$ comes equipped with a degree-2 morphism $\pi: X \ra \mb{P}^1$. It is convenient to assume the projection $\pi$ is ramified over $\infty \in \mb{P}^1$, i.e., that $X$ is an {\it odd} hyperelliptic curve; away from $\infty$, then, $X$ is given by an affine equation
\begin{equation}\label{affine_hyperelliptic_equation}
y^2= f(x)
\end{equation}
in which $\deg_x(f)=2g+1$ and $\pi$ becomes the projection $(x,y) \mapsto x$. As a matter of notation, we also set $\infty_X:= \pi^{-1}(\infty)$.

\medskip
Algebraic morphisms from $X$ to $\mb{P}^r$ of degree $d$ are specified by linear series $(L,V)$, where $\deg(L)=d$ and $V \sub H^0(X,L)$ is an $(r+1)$-dimensional subspace of global sections; in classical notation, this data comprises a $g^r_d$. In this paper, we study the inflection of {\it complete} linear series, i.e., those of the form $V=H^0(X,\mc{O}(\ell \pi^* \infty))$, where $\ell$ is a positive integer, and of certain distinguished subseries arising from the degeneration-based analysis over $\mb{C}$ and $\mb{R}$ carried out in \cite{BCG}. 

\medskip
The line bundle $L=\mc{O}(\ell \pi^* \infty)= \mc{O}(2\ell \infty_X)$ is then relatively orientable, and it has a monomial basis of global sections given by
\begin{equation}\label{eq:monomial_basis}
(1,x,\dots,x^{\ell};y, xy, \dots, x^{\ell-g-1}y)
\end{equation}
in which we suppress monomials divisible by $y$ whenever $\ell \leq g$.

\medskip
Associated with every $g^r_d$ on $X$ there is a natural evaluation morphism \eqref{evaluation_morphism} given by expanding sections locally as Taylor series to order $r$; and when $F$ is of characteristic {\it zero}, there are finitely many {\it inflection} points $p \in X$ along which the evaluation morphism fails to be surjective.\footnote{Strictly speaking, the inflection point $p$ might only be defined over $\ov{F}$.} In positive characteristic, ensuring that the inflection locus is finite and/or easily computable may require additional hypotheses; here our basic policy will be to impose these hypotheses as the need arises. 

\medskip
We now describe in more detail how an inflectionary divisor is defined as a vanishing locus of a section of the determinant $\det(J^{r+1}(L))$ of the corresponding bundle $J^{r+1}(L)$ of principal parts in \eqref{evaluation_morphism}, which is essential for computing local inflectionary indices (see \S \ref{subsec:infl_jet} and \S \ref{subsec:Hasse_Wronskian} for a detailed construction and local analysis of $J^{r+1}(L)$). To begin, note that a given basis $\lambda = (\lambda_0,\ldots,\lambda_r)$ of $V\in\mathrm{Gr}(r+1,H^0(L))$ induces a canonical isomorphism $\psi_{\la}: V \otimes \mc{O} \rightarrow \mc{O}^{r+1}$ that sends $\la=(\la_0,\dotsc,\la_r)$ to the standard coordinate basis $(e_0,\dotsc,e_r)$ of $\mc{O}^{r+1}$. This, in turn, induces an isomorphism $\psi_{\la}^*: \mathcal{H}\mathrm{om}(\mc{O}^{r+1},J^{r+1}(L)) \rightarrow \mathcal{H}\mathrm{om}(V \otimes \mc{O},J^{r+1}(L))$ between vector bundles. The evaluation morphism \eqref{evaluation_morphism} defines a section $\mathrm{ev} \in \mathrm{Hom}(V \otimes \mc{O},J^{r+1}(L))=H^0(\mathcal{H}\mathrm{om}(V \otimes \mc{O},J^{r+1}(L)))$; we define 
$W(\lambda) \in \mathrm{Hom}(\mc{O}^{r+1},J^{r+1}(L))$ of $\lambda$ to be $(\psi_{\la}^*)^{-1}(\mathrm{ev})$. 
We call the determinant $w(\la) \in H^0(\det(J^{r+1}(L)))$ of $W(\la)$ the {\it Hasse Wronskian} of $\la$. The vanishing subscheme of $w(\la)$, i.e., the associated inflectionary locus, is independent of the choice of $\la$; indeed, replacing $\la$ by another basis has the effect of rescaling $w(\la)$ by a scalar.

\medskip
Hereafter, we will refer to the {\it arithmetic inflection divisor} of a $g^r_d$ on $X$ (and defined over $F$) as the class in $\text{GW}(F)$ of the $\mb{A}^1$-Euler class of $\det(J^{r+1}(L))$, whenever the nonsurjective locus of \eqref{evaluation_morphism} is finite and $\det(J^{r+1}(L))$ is relatively orientable. In light of the two preceding paragraphs, this is the case for the complete series $|2 \ell \infty_X|$ on a hyperelliptic curve $X$ ramified over an $F$-rational point $\infty_X$. The arithmetic inflection divisor is a sum of local Euler {\it inflectionary} indices computed by Hasse Wronskians, which we examine next.

\subsection{Nisnevich coordinates}\label{subsec:Nis}
To calculate local inflectionary indices, we carry out calculations in
{\it Nisnevich} local coordinates of a hyperelliptic curve $X$. Here a Nisnevich chart, or system of Nisnevich local coordinates, at a closed point $p \in X$ is an \'etale morphism $\varphi: U \ra \mb{A}^1$ from an open neighborhood $U$ of $X$ containing $p$, which induces an isomorphism $k(\varphi(p)) \cong k(p)$ between residue fields via the pullback $\varphi^*$. We will refer to the latter requirement as the {\it Nisnevich condition}. Note that given a closed point $p \in X$, a Nisnevich local coordinate exists by \cite[Proposition 20]{KW2}. 

\medskip
We then need to check that the local Euler indices obtained by the choices of Nisnevich local coordinates and local trivializations of $\mc{E}=\mc{L}^{\otimes 2} \otimes K_X$ as above are consistent with the relative orientation described in \S \ref{subsec:Rel_Orient}. 
To explain how this works, assume that $X$ is a smooth projective variety of dimension $d>0$, that $\mathcal{E}$ is a vector bundle of rank $d$ on $X$, and that $(X,\mathcal E)$ is relatively oriented, i.e., equipped with a (fixed) isomorphism $\mc{H}\mathrm{om}(\det T_X,\det \mc{E}) \rightarrow \mc{L}^{\otimes 2}$. Given a Nisnevich chart $X \supset U \rightarrow \mb{A}^d$, let $\psi: \mc{E}|_U \rightarrow \mc{O}_U^d$ denote the corresponding trivialization of $\mc{E}$. Our Nisnevich chart and the associated trivialization $\psi$ of $\mc{E}$ determine a \emph{distinguished basis} $(e_1,\dotsc,e_d)$ on $H^0(\mc{E}|_U)$ induced by pulling back standard coordinate bases on $H^0(\mc{O}_U^d)$ under the isomorphism $\psi$. On the other hand, the fact that $\phi$ is \'etale implies that $T_X|_U \cong \phi^*T_{\mb{A}^d}$, so the standard basis $(\partial_1,\dotsc,\partial_d)$ on $H^0(T_{\mb{A}^d})$ pulls back to a \emph{distinguished basis} $(v_1,\dotsc,v_d)$ on $H^0(T_X|_U)$. Similarly, the \emph{distinguished basis} on $\det \mc{E}|_U$ (resp. $\det T_X|_U$) refers to $e_1 \wedge \dotsb \wedge e_d \in H^0(\det \mc{E}|_U)$ (resp. $v_1 \wedge \dotsb \wedge v_d \in H^0(\det T_X|_U)$). The following {\it Nisnevich compatibility} condition was introduced in \cite[Def. 19]{KW2}.


\begin{dfn}\label{def:compatible}
    Given a Nisnevich chart $\phi$ and a local trivialization $\psi$ as above, $\psi$ is compatible with $\phi$ if the element of $\mc{H}\mathrm{om}(\det T_X|_U,\det \mc{E}|_U)$ sending the distinguished basis $v_1 \wedge \dotsb \wedge v_d$ of $\det T_X|_U$ to the distinguished basis $e_1 \wedge \dotsb \wedge e_d$ of $\det \mc{E}|_U$ is a square in $H^0(\mc{L}|_U^{\otimes 2})$, under the identification $\mc{H}\mathrm{om}(\det T_X,\det \mc{E}) \rightarrow \mc{L}^{\otimes 2}$ prescribed by the relative orientation.
\end{dfn}

In practice, it is equivalent to verify the compatibility conditions of various local trivializations by checking that the transition maps of $\mc{H}\mathrm{om}(\det T_X,\det \mc{E})$ between pairs of Nisnevich charts and local trivializations are squares; see, e.g., \cite[Ex. 31]{KW2}. This is because the collection of such trivializations determines the identification $\mc{H}\mathrm{om}(\det T_X,\det \mc{E}) \rightarrow \mc{L}^{\otimes 2}$ up to squares of scalars.

\subsection{Local inflectionary indices}\label{subsec:loc_Euler_general}
In \autoref{subsec:loc_glob_Euler}, we saw that the global Euler class is a sum of local Euler indices. Often, 
local Euler indices encode subtle geometric data. Here, we explain first how to compute local Euler indices 
along an open neighborhood $U$ of a closed point $p \in \mb{A}^1_F$. The reduction from the general case to the affine case is explained in \cite[\S 4]{KW2}, which we summarize now.

\medskip
To compute a local Euler index, suppose that $\sigma$ is a nontrivial section of a relatively orientable line bundle $\mc{E}$ on a curve $X$ defined over $F$, and that $\sig$ vanishes at an isolated closed point $p$ in $X$. Let $\varphi: U \rightarrow \mb{A}^1_F$ be a Nisnevich chart at $p \in U$, and let $\psi:\mc{E}_U \rightarrow \mc{O}_U$ be a compatible local trivialization. Without loss of generality we may assume, shrinking $U$ if necessary, that $p$ is the unique point of $U$ in which $\sigma$ vanishes. Let $Z$ denote the intersection of the closed subscheme $(\sigma=0)$ and $U$; the local ring $\mc{O}_{Z,p}$ is then isomorphic to the quotient $\mc{O}_{U,p}/f$, where $f=\psi(\sigma|_U)$. According to \cite[Proof of Lem. 24]{KW2}, the pullback $\varphi^*:F[x]_{\varphi(p)} \rightarrow \mc{O}_{U,p}$ on local rings now induces surjections $F[x] \twoheadrightarrow \mc{O}_{U,p}/\mf{m}^n$ for every $n \in \mb{N}$, where $\mf{m}$ is the maximal ideal of $\mc{O}_{U,p}$. 
As $\mc{O}_{Z,p}$ is an Artinian local ring, there is some $n \in \mb{N}$ with $\mf{m}^n=0$ in $\mc{O}_{Z,p}$; Kass and Wickelgren then choose some $g \in F[x]$ for which $\varphi^*(g)-f \in \mf{m}^{2n}$, so that $\mc{O}_{Z,p} \cong F[x]_{\varphi(p)}/g$. 
They show in the remainder of \cite[\S 4]{KW2} (see, in particular, \cite[Def. 30 and Cor. 31]{KW2}) that the local Euler index at $p$ may be obtained directly from the presentation $\mc{O}_{Z,p} \cong F[x]_{\varphi(p)}/g$, and is independent of the choice of local Nisnevich coordinate $\varphi$, compatible local trivialization $\psi$, and local section $g \in F[x]$. 
This, in turn, allows us to reduce to the case in which $p \in U = \mb{A}^1_F$ and $\sigma$ is represented as a section $g \in H^0(\mc{O}_{\mb{A}^1_F})=F[x]$ that vanishes in an isolated closed point $p$. 
Note that when the residue field of $p \in \mb{A}^1_F$ is $F$, we may assume without loss of generality that $p=0$. The following proposition characterizes the corresponding local Euler index.



\begin{lemma}\label{lem:locEuler}
    Let $F$ be a field of characteristic $\neq 2$. Suppose that $\sigma \in H^0(\mathcal{O}_{\mathbb{A}^1_F})$ is given by $x^m(a+xg(x))$, for some $a \in F^*$, $m \in \nn$ and $g(x) \in F[x]$; then
    \begin{align*}
        \mathrm{ind}_0\sigma=
        \begin{cases}
            \dfrac{m}{2}\cdot \mb{H} & \text{if $m$ is even}\\
            \noalign{\vskip9pt}
            \dfrac{m-1}{2}\cdot \mb{H} +\left<a\right> & \text{otherwise}.
        \end{cases}
    \end{align*}
\end{lemma}

\begin{proof}
    Applying \cite[Cor. 8]{KW3}, 
    it suffices to find the class in ${\rm GW}(F)$ of a bilinear form associated to the local Newton matrix $\mathrm{New}(x^m(a+xg(x)),0)$ from \cite[Def. 7]{KW3}; since $x^m(a+xg(x))$ is equal to a unit times $ax^m$ in the local ring $F[x]_{(x)}$, \cite[Def. 7]{KW3} implies that $\mathrm{New}(x^m(a+xg(x)),0)=\mathrm{New}(ax^m,0)$ is the following $m$ by $m$ matrix:
    \begin{align*}
        \begin{pmatrix}
            0 & 0 & \cdots & 0 & a^{-1}\\
            0 & 0 & \cdots & a^{-1} & 0\\
            \vdots & \vdots & \reflectbox{$\ddots$} & \vdots & \vdots\\
            0 & a^{-1} & \cdots & 0 & 0\\
            a^{-1} & 0 & \cdots & 0 & 0
        \end{pmatrix}.
    \end{align*}
    By \cite[Lemma 6]{KW3}, the class in ${\rm GW}(F)$ of $\mathrm{New}(ax^m,0)$ matches with the statement of Lemma~\ref{lem:locEuler} except that $a$ is replaced with $a^{-1}$; we conclude using $\left<a^{-1}\right>=\left< a^2 \cdot a^{-1} \right>=\left<a\right>$.
\end{proof}

\begin{rem}\label{rem:loc_Euler_cases}
    When $F=\mathbb{R}$, we have $\left<a\right>=\left<1\right>$ if $a>0$, and $\left<a\right>=\left<-1\right>$ otherwise. In this case, the local index recovers Milnor's real oriented index. Indeed, suppose that we have a smooth function $f:\mathbb{R}^n\to\mathbb{R}^n$ and $q\in\mathbb{R}^n$ is a regular value of $f$ (Sard's theorem implies that such a regular value exists). The inverse function theorem guarantees that for any point $p\in f^{-1}(\{q\})$, the map $f$ is a diffeomorphism between a neighborhood of $p$ and $q$. The Milnor index at the point $p$ will be $1$ if the restriction of $f$ to that neighborhood is orientation-preserving, and $-1$ if it is orientation-reversing. In particular, when $m=1$, we have a map $\sigma:\mathbb{R}\to\mathbb{R}$, $\sigma(x) = ax + g(x)x^2$ that has a regular value at $0$, and its derivative at the origin is $\sigma^{\pr}(0) = a$; the sign of $a$ determines whether or not $\sigma$ is orientation-preserving.
    
    \medskip
    Now say $F=\mathbb{F}_q$, where $2 \not | \; q$. Recall that ${\rm GW}(\mathbb{F}_q) \cong \mathbb{Z} \times \mathbb{F}_q^{\times}/(\mathbb{F}_q^{\times})^2
    \cong \mb{Z} \times \mathbb{Z}/2\mathbb{Z}$. Thus $\left<a \right>=\left<1\right>$ if and only if $a$ is a square in $\mb{F}_q$. 
    In particular, $\left<-1\right>=\left<1\right>$ if and only if $-1$ is a square in $\mathbb{F}_q$.
    
    \medskip
    A significant and interesting case arises when $q=p$ is an odd prime. Then $\left< -1\right>=\left<1\right>$ if and only if $p \equiv 1$ modulo 4, while $\left<a\right>=\left<1\right>$ if and only if $\left(\dfrac{a}{p}\right)=1$, where $\left(\dfrac{a}{p}\right)$ is the Legendre symbol. For example, when $p \equiv 1$ modulo $4$, quadratic reciprocity reduces Lemma~\ref{lem:locEuler} to the statement that
    \begin{align*}
        \mathrm{ind}_0\sigma=
        \begin{cases}
            m\left<1\right> & \text{if $m$ is even}\\
            \noalign{\vskip9pt}
            (m-1)\left<1\right> +\left<a\right> & \text{otherwise}.
        \end{cases}
    \end{align*}
\end{rem}

\medskip

In general, an isolated zero $p$ of a section $\sigma$ may not be defined over the base field $F$. Nonetheless, under a mild hypothesis on the residue field of $p$, the local Euler index may be computed using \cite[Theorem 1.3]{BBMMO}, as follows.

\begin{lemma}\label{lem:locEuler_bch}
    Let $F$ be a field. If $\sigma \in H^0(\mathcal{O}_{\mathbb{A}^1_F})$ has an isolated zero at $p \in \mathbb{A}^1_F$ and $k(p)/F$ is separable, then 
    \[
        \mathrm{ind}_p\sigma=\mathrm{Tr}_{k(p)/F}\mathrm{ind}_{\overline{p}}(\sigma \otimes_F 1)
    \]
    where $\overline{p}$ is a point above $p$ in $\mathbb{A}^1_{k(p)}$, $\sigma \otimes_F 1$ is the lift of $\sigma$ as a section of $\mc{O}_{\mb{A}^1_{k(p)}}$ via the base change into $k(p)$, and $\mathrm{Tr}_{k(p)/F}: {\rm GW}(k(p)) \ra {\rm GW}(F)$ denotes the trace on bilinear forms induced via post-composition by the field trace ${\rm tr}_{k(p)/F}:k(p) \ra F$.\footnote{In fact, $\mathrm{ind}_{\overline{p}}(\sigma \otimes_F 1)$ is independent of the choice of $\overline{p}$ because $\sigma \otimes_F 1$ is fixed under the action of $\mathrm{Gal}(k(p)/F)$.}
\end{lemma}

\begin{rem}\label{rmk:geom_pt}
    When the residue field $k(x)$ of a point $x \in X$ is a trivial extension of $F$, 
    the dual $dv$ of the distinguished basis $v$ of $T_X|_U$ 
    associated with a local Nisnevich 
    chart $\varphi: U \rightarrow \mb{A}^1_F \cong \mathrm{Spec}\: F[t]$ coincides with the pullback $\varphi^*(dt)$, where $t$ up to $F$-translation is a uniformizer of $p:=\varphi(x) \in \mathrm{Spec}\: F[t]$; we call $dv$ the distinguished basis of $K_X|_U$. This fact is quite useful 
    in light of 
    Definition~\ref{def:compatible}, and given the role played by the canonical bundle $K_X$ in constructing the line bundle $\mc{E} = \det(J^{r+1}(L)) \cong L^{\otimes r+1} \otimes K_X^{\binom{r+1}{2}}$. 
    
    \medskip
    However, the distinguished basis $dv=\varphi^*(dt)$ no longer coincides with the pullback of the differential of a uniformizer of $p$ when the extension $k(x)/F$ is nontrivial: there is no $F$-linear function of $t$ 
    with a simple zero in $p$. 
    As a consequence of Lemma~\ref{lem:locEuler_bch}, the computation of the local Euler index of a nontrivial section $\sigma \in H^0(\mc{E})$ at $x$ reduces to that of a $k(p) \cong k(x)$-rational closed point $\overline{p} \in \mathrm{Spec}\: k(p)[t]$ above $p:=\varphi(x) \in \mathrm{Spec}\: F[t]$. Observe that $\overline{p}$ is an image $\overline{\varphi}(\overline{x})$ of a $k(x)$-rational point $\overline{x} \in X \times_{\mathrm{Spec}(F)} \mathrm{Spec}(k(x))$ above $x$ under the lift $\overline{\varphi}:U \times_{\mathrm{Spec}(F)} \mathrm{Spec}(k(p)) \rightarrow \mathrm{Spec}\: k(p)[t]$ of $\varphi$. The main advantage of using a $\mathrm{Gal}(k(x)/F) \cong \mathrm{Gal}(k(p)/F)$-equivariant local Nisnevich coordinate 
    for $X \times_{\mathrm{Spec}(F)} \mathrm{Spec}(k(p))$ is that the $\overline{\varphi}^*(dt)$ is indeed the dual $d\overline{v}$ of the distinguished basis $\overline{v}$ under the local Nisnevich 
    chart $\overline{\varphi}$, where $\overline{v}$ is the lift of $v$ under the base change $k(p)/F$.
    
    \medskip
    Hereafter, we abusively conflate a given local Nisnevich chart 
    and compatible local trivialization $\psi$ at $x$ 
    with their corresponding lifts under the base change $k(x)/F$ whenever the field extension $k(x)/F$ is nontrivial. With this abuse of notation, calling $t$ a \emph{local Nisnevich uniformizer at $x$} whenever $\mb{A}^1_F \cong \mathrm{Spec}\: F[t]$ 
    is justified up to $k(x)$-translation. 
\end{rem}

\subsection{A global arithmetic Euler class}\label{subsec:global_Euler_general}

Here we present a general formula for the global $\mb{A}^1$-Euler class of a relatively orientable line bundle on a curve; as a consequence, we deduce a global arithmetic inflection formula for complete linear series $|2\ell \infty_X|$ on a hyperelliptic curve $X$.

\begin{thm}\label{thm:global_Euler_class}
    Suppose that $\mc{E}$ is a relatively orientable line bundle on a curve $X$ defined over a field $F$ not of characteristic 2, equipped with a nontrivial section $\sigma \in H^0(\mc{E})$ that vanishes in isolated points. 
    The $\mb{A}^1$-Euler class $e(\mc{E})$ in ${\rm GW}(F)$ is $\frac{m_{\overline{F}}}{2} \cdot \mb{H}$, where $m_{\overline{F}}$ is the degree of $\mc{E}$.\footnote{Here the fact that $\mc{E}$ is relatively orientable forces $m_{\overline{F}}$ to be even.}
\end{thm}

\begin{proof}
    The statement is a special case of \cite[Prop. 19]{SW}; here we give the salient points of its proof. As explained in section~\ref{subsec:loc_glob_Euler}, \cite[Thm 1.1]{BW} establishes that the global Euler class $e(\mc{E})$ is a sum of local Euler indices $\mathrm{ind}_p\sigma$ over points $p \in X$ where $\sigma$ vanishes. Moreover, {\it loc.cit.} implies that $e(\mc{E})$ is independent of the choice of section $\sigma$.
    
    \medskip
    Now, given a field extension $L/F$, let $\mc{E}_L$ (resp. $\sigma_L$) denote the pullback of $\mc{E}$ (resp. $\sigma$) under the base change $X_L \rightarrow X$. Applying Lemma~\ref{lem:locEuler} together with the fact that $\left<\al\right>\mb{H}=\mb{H}$ for every $\al \in L^*$ (see \cite{Lam}), we have $\mathrm{ind}_x(\alpha\sigma_L)=\left<\al\right> \mathrm{ind}_x(\sigma_L)$. 
    Comparing the sums of local Euler indices for $\sigma_L$ and $\al \cdot \sigma_L$, we immediately see that $e(\mc{E}_L)=\left<\al\right>e(\mc{E}_L)$. Note that the field extension $L/F$ induces a map on Grothendieck--Witt groups ${\rm GW}(F) \rightarrow {\rm GW}(L)$, which identifies $e(\mc{E}_L)$ with the image of $e(\mc{E})$. 
    
    \medskip
    The fact that $e(\mc{E})$ is an integer multiple of $\mb{H}$ 
   follows from basic results of Milnor--Witt theory as developed by Morel in \cite[Ch. 3]{Mor}. Specifically, there is a (residue) map $\partial_t: {\rm GW}(F(t)) \rightarrow {\rm W}(F)$ by \cite[Lem. 3.10 and Theorem 3.15]{Mor}, 
    and the map $\eta: {\rm GW}(F) \rightarrow {\rm W}(F)$ which coincides with the quotient ${\rm GW}(F) \rightarrow {\rm GW}(F)/\mb{H}$. It then follows that
    \[
        0=\partial_t(e(\mc{E}_{F(t)}))=\partial_t(\left<t\right>e(\mc{E}_{F(t)}))=\eta(e(\mc{E}))
    \]
    by \cite[Remark 21]{SW} and the fact that $e(\mc{E}_{F(t)})=\left<\al\right>e(\mc{E}_{F(t)})$ for any $\al \in F(t)^*$ as above. The fact that $\eta(e(\mc{E}))=0$ immediately implies that $e(\mc{E})$ must be an integer multiple of $\mb{H}$. 
    
    \medskip
    Finally, to see that $e(\mc{E})$ is exactly $\frac{m_{\overline{F}}}{2} \cdot \mb{H}$, it suffices to show that $\mathrm{rk}(e(\mc{E})) = m_{\overline{F}}$, where $\mathrm{rk}:{\rm GW}(F)\to\mathbb{Z}$ is 
    induced by the rank map on bilinear forms. Indeed, this is enough because $\mathrm{rk}(\mb{H})=2$. To see it, note that the divisor $(\sigma=0)$ has degree $m_{\overline{F}}$ by definition. On the other hand, every local Euler index $\mathrm{ind}_p(\sigma)$ for each $p \in (\sigma=0)$ has rank equal to the multiplicity of $(\sigma=0)$ at $p$. Summing ranks of local Euler indices, we see that the rank of $e(\mc{E})$ must be $m_{\overline{F}}$, and we conclude.
\end{proof}


\begin{coro}\label{thm:global_Euler_class_inflection}
Let $F$ be a field with $\mathrm{char}(F) \neq 2$, $\ell \geq 1$ a positive integer, and $L=\mc{O}(2\ell \infty_X)$, where $\pi:X \ra \mb{P}^1$ is a hyperelliptic curve of genus $g$ defined over $F$ and ramified in $\infty_X=\pi^{-1}(\infty)$. 
Assume, moreover, that 
the number of $\ov{F}$-inflection points is finite. Associated to the complete linear series $|L|$ on $X$ there is a well-defined arithmetic $F$-inflection class in ${\rm GW}(F)$ given by
\[
[\mbox{\rm{Inf}}_F(|L|)]_{\mb{A}^1}=\frac{\ga_{\mb{C}}}{2}\cdot \mb{H}
\]
where 
\begin{equation*}
\ga_{\mb{C}}:= \begin{cases} \ell(\ell+1)(g+1) & \text{if } \ell  \le g \\ g(2\ell-g+1)^2 & \text{if } \ell > g \end{cases}
\end{equation*}
is the $\mb{C}$-inflectionary degree computed by the usual Pl\"ucker formula.
\end{coro}

\begin{proof}
The line bundle $\mc{E}=\det J^{r+1}(L)$ is isomorphic to $L^{\otimes (r+1)} \otimes K_X^{\otimes \binom{r+1}{2}}$ by Proposition~\ref{prop:jet_exact_seq}. As $K_X \cong \mc{O}((2g-2)\infty_X)$ and $L=\mc{O}(2\ell \infty_X)$ are both tensor squares over $F$, $\mc{E}$ on $X$ is relatively orientable. Moreover, from section~\ref{subsec:lin_series_hypell}, $\mc{E}$ admits a Hasse Wronskian section $w(\la)$, whose vanishing locus comprises the zero-dimensional $\ov{F}$-inflection divisor on $X$ whenever the number of $\ov{F}$-inflection points is finite. The associated arithmetic $F$-inflection class is the $\mb{A}^1$-Euler class of $\mc{E}$. According to Theorem~\ref{thm:global_Euler_class}, this is $\frac{m_{\ov{F}}}{2} \cdot \mb{H}$, where
\[
m_{\ov{F}}= \deg \left(L^{\otimes (r+1)} \otimes K_X^{\otimes \binom{r+1}{2}}\right)= \begin{cases} \ell(\ell+1)(g+1) & \text{if } \ell  \le g \\ g(2\ell-g+1)^2 & \text{if } \ell > g \end{cases}
\]
which in view of \eqref{eq:monomial_basis} is precisely $\ga_{\mb{C}}$. 
\end{proof}

\subsection{Toric geometry of hyperelliptic curves}\label{subsec:toric_geometry}

\medskip
Any fact about toric varieties in this subsection is standard; we follow \cite{CLS}. As before, let $X$ denote a hyperelliptic curve of genus $g$ defined over a field $F$ with $\text{char}(F) \neq 2$. Whenever $g \geq 2$, passing to the projective completion of $X:y^2=f(x)$ by homogenizing the affine equation of $X$ with respect to an auxiliary variable introduces singularities. On the other hand, $X$ always smoothly embeds in a Hirzebruch surface, namely $\mb{F}_d=\mb{P}(\mc{O}_{\mb{P}^1} \oplus \mc{O}_{\mb{P}^1}(\mc{D}))$, where $2\mc{D} \sim B$ is the branch divisor of the hyperelliptic structure map $\pi$ and $d=\deg \mc{D}=g+1$. The cone of curves on $\mb{F}_d$ is generated by (the classes of) a fiber $f$ and by the unique section $\sig$ with self-intersection $-d$. The adjunction formula implies that $X \sim 2(\sigma+df)$, so in particular $X$ and $\sigma$ do not intersect. Accordingly, it suffices to work locally over those two toric local coordinate charts $U_{\tau_1}$ and $U_{\tau_2}$ of $\mathbb{F}_d$ corresponding to the cones $\tau_1$ and $\tau_2$, as in Figure~\ref{fig:Hirzebruch_fan}.

\begin{figure}
\begin{tikzpicture}[scale=0.7]
\draw[fill=gray!50] (0,2) -- (0,0) -- (2,0) (2,2) node{$\tau_1$};
\draw[fill=gray!10] (-1,-2) -- (0,0) -- (0,2) (-2,0) node{$\tau_2$};
\draw (0,0) -- (0,-2);
\end{tikzpicture}
\caption{The fan of $\mathbb{F}_d$, with rays $u_1=e_1$, $u_2=e_2$, $u_3=-e_1-de_2$, $u_4=-e_2$. Here $\tau_1= \mathrm{Span}\{u_1,u_2\}$, $\tau_2= \mathrm{Span}\{u_2,u_3\}$, and $d=2$.}
\label{fig:Hirzebruch_fan}
\end{figure}
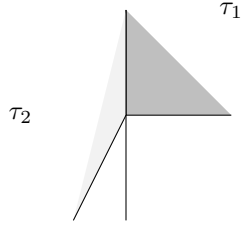

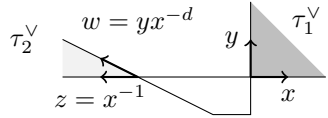
\begin{figure}
\begin{tikzpicture}[scale=0.5]
\draw[fill=gray!50] (0,2) -- (0,0) -- (2,0) (1.5,1.5) node{$\tau_1^\vee$};
\draw (0,0) -- (-3,0) -- (-1,-1) -- (0,-1) -- (0,0);
\draw[fill=gray!10] (-5,1) -- (-3,0) -- (-5,0) (-6,1) node{$\tau_2^\vee$};
\draw[thick,->] (0,0)--(1,0);
\draw (1,-0.5) node{$x$};
\draw[thick,->] (0,0)--(0,1);
\draw (-0.5,1) node{$y$};
\draw[thick,->] (-3,0)--(-4,0);
\draw (-4,-0.5) node{$z=x^{-1}$};
\draw[thick,->] (-3,0)--(-4,0.5);
\draw (-3,1.5) node{$w=yx^{-d}$};
\end{tikzpicture}
\caption{The dual polygon of $\mb{F}_d$, along with the dual cones $\tau_1^{\vee}$ and $\tau_2^{\vee}$ generated by rays $e_1$, $e_2$, $-e_1$, $-de_1+e_2$, which determine characters $x$, $y$, $z$, $w$ respectively. Here $\tau_1^{\vee}=\mathrm{Cone}(e_1,e_2)$ and $\tau_2^{\vee}=\mathrm{Cone}(-e_1,-de_1+e_2)$.}
\label{fig:dual_Hirzebruch_fan}
\end{figure}


\medskip
Here we use the canonical stratification of $\mb{F}_d$ by torus orbits prescribed by a {\it fan} $\Sig$ contained inside the Euclidean space $N_{\mb{R}}= N \otimes_{\mb{Z}} \mb{R}$ determined by the lattice $N \cong \mb{Z}^2$ of one-parameter subgroups $F^* \ra T$, where $T$ denotes the maximal torus of $\mb{F}_d$. 
Each two-dimensional cone corresponds uniquely to a torus-fixed point of $\mb{F}_d$, and each ray $r$ corresponds to a $T$-invariant curve $\mathrm{D}_r$, whose closure contains torus-fixed points corresponding to adjacent two-dimensional cones; in our case there are four of these, namely $\mathrm{D}_{u_i}$, $i=1,2,3,4$. The lattice $N$ of one-parameter subgroups is dual to the lattice $M$ of {\it characters} $\chi:T \ra F^*$; and $M_{\mb{R}}= M \otimes_{\mb{Z}} \mb{R}$ contains a polygon $\mc{P}$ (in the case at hand, a trapezoid) that is dual to $\Sig$.



\medskip
By definition, the ring of functions of each affine open $U_{\tau_i}$ is the semigroup algebra of the dual cone $\tau_i^{\vee}$ of linear functionals that pair nonnegatively with $\tau_i$, $i=1,2$. See Figure~\ref{fig:dual_Hirzebruch_fan}, which illustrates the polygon that underlies $\mb{F}_d$, each of whose vertices corresponds to a cone of the dual fan; the characters determined by each cone's rays give canonical local affine coordinates.\footnote{More precisely, the fan of $\mb{F}_d$ may be recovered as the collection of outer normals to the edges of the polygon, translated so that they collectively emanate from the origin.}
It follows that there are canonical isomorphisms $\psi_1: U_{\tau_1} \xrightarrow{\sim} \mathbb{A}^2_{x,y}$ and $\psi_2: U_{\tau_2} \xrightarrow{\sim} \mathbb{A}^2_{z,w}$. 
Via $\psi_1$, $\mathrm{D}_{u_1}$ and $\mathrm{D}_{u_2}$ are identified with $(y=0)$ and $(x=0)$ in $\mb{A}^2_{x,y}$, respectively. Similarly, via $\psi_2$, $\mathrm{D}_{u_2}$ and $\mathrm{D}_{u_3}$ correspond to $(z=0)$ and $(w=0)$, respectively. Since $z=x^{-1}$ and $w=x^{-d}y$, the transition map $\varphi: U_{\tau_1}\backslash \mathrm{D}_{u_2} \rightarrow U_{\tau_2}\backslash \mathrm{D}_{u_2}$ is identified with $(\psi_2\circ \varphi\circ \psi_1^{-1})(x,y)=(x^{-1},yx^{-d})$. 

\medskip
Now let $\theta:=\sigma+df$. Since $X$ is given by a vanishing locus of a section of 
$\mathcal{O}(2\theta)$, we also need to trivialize the bundle $\mathcal{O}(\theta)$ over $U_{\tau_2}$. Since $\mathrm{D}_{u_2} \sim \theta$, it is easy to see that the transition map of $\mathcal{O}(\theta)$ corresponding to $\varphi$ is given by multiplication by $z^d$. Similarly, the transition map for $\mathcal{O}(f)$, which restricts to $\mathcal{O}_X(2\infty)$, is given by multiplication by $z$. It follows from the preceding discussion that $X$ is given in $U_{\tau_2} \cong \mb{A}^2_{z,w}$ by $w^2=h(z)$, where $h(z):=z(z^{2g+1}f(z^{-1}))$.

\subsection{Nisnevich coordinates and linear projections}\label{subsec:nisnevich_coords_and_projections}
Recall 
that in order to compute local Euler indices, we first need to specify local Nisnevich coordinates with compatible local trivializations; see section~\ref{subsec:compute_loc_Euler} and Definition~\ref{def:compatible}.
Here our aim is to construct, under mild assumptions, explicit compatible local Nisnevich coordinates for complete linear series associated to line bundles $L=\mc{O}(2\ell\infty_X)$ on hyperelliptic curves $X$ as in section~\ref{subsec:lin_series_hypell}.

\medskip
To begin, let $p \in X \setminus \{ \infty_X \}$. 
Assuming that the base field $F$ is perfect, the restriction of a general linear projection $\mb{A}^2_{x,y} \rightarrow \mb{A}^1_t$ to $X \setminus \{\infty_X\} \subset U_{\tau_1} \cong \mb{A}^2_{x,y}$, where $t$ is a nontrivial $F$-linear combination of $x$ and $y$,
determines a local Nisnevich chart at $p$.

\begin{rem}\label{rmk:perfect_field}
    Perfectness of the base field $F$ is necessary in order to guarantee that we can cover a given hyperelliptic curve $X$ by local Nisnevich charts. Indeed, given $p \in X \setminus \{ \infty_X \}$, 
    the fact that $\mathrm{char}(F) \neq 2$ ensures that the hyperelliptic projection $\pi: X \rightarrow \mb{P}^1$ is separable over $x=\pi(p) \in \mb{A}^1$. Therefore, it suffices to check that $k(x)/F$ is separable. This is automatic whenever $F$ is perfect. Identifying $p$ with $(x,y) \in \mb{A}^2 \cong U_{\tau_1} \supset X \setminus \{\infty_X\}$, we have $k(p)=k(x,y)=k(ax+by)$ for some $a,b \in F$ because $k(x,y)/F$ is separable. 
    If $F$ is not perfect, our basic strategy for producing local Nisnevich coordinates via linear projection only works provided 
    all residue fields of inflection points on $X$ are separable extensions of $F$. 
    For example, in Remark~\ref{rmk:infl_poly_cond}, if a point $\ga \in \mb{A}^1_x \subset \mb{P}^1$ is inert with respect to the hyperelliptic projection $\pi:X \rightarrow \mb{P}^1$, checking the separability of the extension $k(\ga,\sqrt{f(\ga)})/F$ is equivalent to checking separability of the extension $k(\ga)/F$. 
\end{rem}

Near $\infty_X$, 
we use the linear projection $\mb{A}^2_{z,w} \rightarrow \mb{A}^1_w$. More precisely, 
we have $U_{\tau_2} \cong \mb{A}^2_{z,w}$, and near $\infty_X \in X \cap U_{\tau_2}$ the restriction of the projection is automatically Nisnevich at $\infty_X$. Therefore, we hereafter assume that the base field $F$ is perfect so that the restrictions of both kinds of projections 
under consideration are Nisnevich in neighborhoods of desired points in $X$.

\medskip
For the local trivializations of $\mc{E} = \det J^{r+1}(L)$ in either case, we first need to fix a local trivialization of $L$, so that we can then use the associated local trivializations on $\mc{E}$ induced by those of $L$ and $K_X$. In order to understand local trivializations of $L$, consider the toric fan associated to $\mb{F}_{g+1} \supset X$ as in \S\ref{subsec:toric_geometry}; our goal is to trivialize $L$ by trivializing a line bundle on $\mb{F}_{g+1}$ that restricts to $L$. For this purpose, note that $L \cong \mc{O}_X(a\sigma + \ell f)$ for any $a$, as $\mc{O}_X(\sigma) \cong \mc{O}_X$ and $\mc{O}_X(f) \cong \mc{O}_X(2\infty_X)$. It therefore suffices to produce an integer $a \in \mb{Z}$ so that $|L|$ is induced by a linear series in $|\mc{O}_{\mb{F}_{g+1}}(a\sigma + \ell f)|$. The appropriate choice, for any $\ell$, is $a=1$ (note that $a=0$ also works when $\ell \le  g$). Indeed, whenever $\ell \le g$, any effective divisor on $\mb{F}_{g+1}$ linearly equivalent to $\sigma + \ell f$ must be a union of $\sigma$ together with $\ell$ fiber classes up to multiplicity, while any section of $L$ arises from $H^0(\mc{O}_{\mb{P}^1}(\ell))$ via pullback under the hyperelliptic structure map $X \rightarrow \mb{P}^1$. If instead $\ell > g$, $|L|$ is generated by $n(\mathrm{D}_{u_1} \cap X) + (\ell - n)(2\infty_X)$ and $R_{\pi} + m(\mathrm{D}_{u_1} \cap X) + (\ell - m - g - 1)(2\infty_X)$, where $0 \le m \le \ell - g - 1$ and $0 \le n \le \ell$ (here $\mathrm{D}_{u_1}$ compactifies $(x=0) \subset \mb{A}^2_{x,y} \cong U_{\tau_1}$, while $2\infty_X = \mathrm{D}_{u_3} \cap X$). Similarly, $\mathrm{D}_{u_2}$ compactifies $(y=0) \subset \mb{A}^2_{x,y} \cong U_{\tau_1}$, while $R_{\pi} = X \cap \mathrm{D}_{u_2}$; the fact that $|\sigma + \ell f|$ restricts to $|L|$ on $X$ follows. 

\medskip
To trivialize $L$ on $X$, it suffices to 
trivialize $\mc{O}_{\mb{F}_{g+1}}(\sigma+\ell f)$ along the open subschemes $U_{\tau_i}, i=1,2$ of $\mb{F}_{g+1} \setminus \mathrm{D}_{u_4} = U_{\tau_1} \cup U_{\tau_2}$, as $U_{\tau_1} \cup U_{\tau_2}$ contains $X$. Observe that $U_{\tau_1} \cup U_{\tau_2}$ is isomorphic to the total space of a line bundle $\mc{O}_{\mb{P}^1}((g+1) \infty)$ 
whose zero section 
corresponds to $\mathrm{D}_{u_1} \subset U_{\tau_1} \cup U_{\tau_2}$; note that $\mathrm{D}_{u_1}$ compactifies $(y=0) \subset \mb{A}^2_{x,y} \cong U_{\tau_1}$. Note that the restriction of $\mc{O}_{\mb{F}_{g+1}}(\sigma+\ell f)$ 
to $U_{\tau_1} \cup U_{\tau_2}$ 
is isomorphic to the pullback $\mc{O}_{U_{\tau_1} \cup U_{\tau_2}}(\ell f)$ of $\mc{O}_{\mb{P}^1}(\ell \infty)$. 
Coordinatizing $\mb{P}^1$ as $\mb{P}^1_{[\al:\be]}$ with $x=\frac \al\be$ and $z=\frac \be\al$ ($\infty$ now corresponds to $[1:0]$), $H^0(\mc{O}_{\mb{P}^1}(\ell \infty))$ becomes identified with $F[\al,\be]_{\ell}$. 
Let $s'_1,s'_2$ denote the sections corresponding to $\be^\ell,\al^\ell$ respectively via this identification; then on $\mb{P}^1_{[\al:\be]} \setminus \{[1:0],[0:1]\}$, we have $\frac{s'_2}{s'_1}=\left(\frac{\al}{\be}\right)^\ell=x^\ell=z^{-\ell}$.\footnote{Note that on the affine part $\mb{A}^1_x = \mb{P}^1_{[\al:\be]} \setminus \{[1:0]\}$, we have $(s_1',s_2')=(1,x^\ell)$.}
Letting $s_i, i=1,2$ denote the pullback of $s^{\pr}_i$ via the projection $\pi:U_{\tau_1} \cup U_{\tau_2} \rightarrow \mb{P}^1_{[\al:\be]}$, we obtain trivializations $\psi_{\mb{F}_{g+1},i}:\mc{O}_{U_{\tau_1} \cup U_{\tau_2}}(\ell f)|_{U_{\tau_i}} \rightarrow \mc{O}_{U_{\tau_i}}$ for which the distinguished basis (defined by pulling back the standard basis of $\mc{O}_{U_{\tau_i}}$; see \S~\ref{subsec:Nis}) is exactly $s_i|_{U_{\tau_i}}$.

\medskip
Now, letting $r+1=h^0(L)$, we may trivialize $\mc{E} = \det(J^{r+1}(L)) \cong L^{\otimes (r+1)} \otimes K_X^{\binom{r+1}2}$ 
and $\mc{H}=\mc{H}\mathrm{om}(T_X,\mc{E}) \cong L^{\otimes (r+1)} \otimes K_X^{\binom{r+1}2+1}$ on $X$ by tensoring trivializations on $L$, $T_X$, and $K_X$ with respect to linear projections. 
Along a Nisnevich chart $\varphi_t: U_t \rightarrow \mb{A}^1_t$ given by the restriction of a linear projection $\mb{A}^2_{x,y} \rightarrow \mb{A}^1_t$ to an open subset $U_t$ of $X \cap U_{\tau_1} \subset U_{\tau_1} \cong \mb{A}^2_{x,y}$\footnote{Here $t$ is a nontrivial $F$-linear combination of $x$ and $y$.}, 
let $\psi_{L,t}:L|_{U_t} \rightarrow \mc{O}_{U_t}$ denote a trivialization of $L$ 
with distinguished basis $s_1|_{U_t}$. 
Let $\psi_{T_{U_t}}:T_{U_t} \rightarrow \mc{O}_{U_t}$ be a trivialization of $T_{U_t}$ that sends the distinguished basis $v_t$ on $T_{U_t}$ singled out by $\varphi_t$ (see section~\ref{subsec:Nis}) to the standard basis of $\mc{O}_{U_t}$; 
and define a corresponding trivialization $\psi_{K_{U_t}}:K_{U_t} \rightarrow \mc{O}_{U_t}$ of $K_{U_t}$ by using the dual basis $dv_t$ as the distinguished basis. Taking tensor powers of these trivializations, we now obtain a trivialization $\psi_{\mc{E},t}$ (resp., $\psi_{\mc{H},t}$) of $\mc{E}$ (resp., $\mc{H}$) along $U_t$ with distinguished basis $s_t := (s_1|_{U_t})^{r+1} \otimes (dv_t)^{\otimes \binom{r+1}2}$ (resp., $s_t \otimes dv_t = (s_1|_{U_t})^{r+1} \otimes (dv_t)^{\otimes \binom{r+1}2+1}$). 
Arguing similarly along a Nisnevich chart $\varphi_w: U_w \rightarrow \mb{A}^1_w$ obtained by restricting the linear projection $\mb{A}^2_{z,w} \rightarrow \mb{A}^1_w$ to an open subset $U_w \subset X \cap U_{\tau_2} \subset U_{\tau_2} \cong \mb{A}^2_{z,w}$ containing $\infty_X$ (which has coordinate $(0,0)$ in $\mb{A}^2_{z,w} \cong U_{\tau_2}$), we construct a trivialization $\psi_{\mc{E},w}$ (resp., $\psi_{\mc{H},w}$) on $\mc{E}$ (resp, $\mc{H}$) along $U_w$ with distinguished basis $s_w := (s_2|_{U_w})^{r+1} \otimes (dv_w)^{\otimes \binom{r+1}2}$ (resp., $s_w \otimes dv_w = (s_2|_{U_w})^{r+1} \otimes (dv_w)^{\otimes \binom{r+1}2+1}$).

\medskip
To check that the local trivializations $\psi_{\mc{E},t}$ and $\psi_{\mc{E},w}$ are compatible (see Definition~\ref{def:compatible}) along their respective Nisnevich 
charts $\varphi_t$ and $\varphi_w$, it suffices to check that the distinguished bases of $\psi_{\mc{H},t}$ and $\psi_{\mc{H}_w}$ are squares. 
Indeed, this is the case because the distinguished bases $s_t \otimes dv_t$ and $s_w \otimes dv_w$ exactly correspond to the maps that send $v_t$ and $v_w$ to $s_t$ and $s_w$, respectively; and Definition~\ref{def:compatible} imposes conditions on such maps. In particular, we seek conditions 
that ensure that both $r+1$ and $\binom{r+1}{2}+1$ are even in order to guarantee the compatibility condition is met.\footnote{Note that $s_1$ and $s_2$ need not themselves be squares.} 
When $\ell \le g$, we have $r+1=\ell + 1$, 
so our condition amounts to requiring that $\ell \equiv 1 \pmod{4}$. Similarly, when $\ell > g$, we have $r+1 = 2\ell-g+1$, so our condition amounts to requiring that $2\ell-g \equiv 1 \pmod{4}$. See Remark~\ref{rmk:compatibility} below for further discussion of these conditions.

\medskip
Sometimes it is useful to remember the transition maps between local trivializations of $\mc{E}$, as it is often easier to represent sections of $\mc{E}$ under a desired local trivialization by using transition maps. Since the local Euler index of a nontrivial section $\sigma \in H^0(\mc{E})$ only depends on the power series representation (see discussions above Lemma~\ref{lem:locEuler}), we can apply this method to find the power series representation of $\sigma$ under a desired Nisnevich local trivialization from a chosen \'etale local trivialization (meaning a choice of \'etale local coordinate together with a compatible local trivialization).
The transition map $\psi_{\mc{E},\tilde t} \circ \psi_{\mc{E},t}^{-1}$ that relates a local trivialization with uniformizer $t$ to another local trivialization with uniformizer $\tilde t$ is multiplication by $\left(\frac{Dt}{d\tilde t}\right)^{\binom{r+1}{2}}$ as $s_t|_{U_t \cap U_{\tilde t}}=\left(\frac{Dt}{d\tilde t}\right)^{\binom{r+1}{2}} s_{\tilde t}$ in $\mc{E}|_{U_t \cap U_{\tilde t}}$. 
Hereafter, a local Nisnevich/\'etale coordinate $t$ means a local Nisnevich/\'etale coordinate with a uniformizer $t$ (which is either a $F$-linear combination of $x$ and $y$, or $t=w$) that is equipped with a compatible local trivialization.

\begin{rem}\label{rmk:compatibility}
    Bachmann and Wickelgren explain in \cite[p.17--18]{BW} that any given local trivialization (without any conditions on $\ell$ and $g$) may be modified to a compatible local trivialization. Doing so involves 
    rescaling the local trivialization by an appropriate scalar function; without precisely identifying this scalar function, however, it is difficult to compute the relevant transition maps explicitly. Consequently, we always assume in this section that either (1) $\ell \le g$ and $\ell \equiv 1 \text{ (mod }4)$, or (2) $\ell > g$ and $2\ell-g \equiv 1 \pmod{4}$, though it would be nice to know whether there is a natural way to overcome this obstacle.
\end{rem}

\subsection{Inflectionary numerology}\label{subsec:numerology}

As we have detailed in the preceding subsections, we work exclusively over fields of characteristic not equal to 2 in order to ensure the existence of affine models $y^2=f(x)$ for hyperelliptic curves; and for the purposes of calculating local Euler indices, we work with complete linear series $|2\ell \infty_X|$ with either (1) $\ell \le g$ and $\ell \equiv 1 \text{ (mod }4)$, or (2) $\ell > g$ and $2\ell - g \equiv 1 \pmod{4}$, to ensure the existence of compatible Nisnevich trivializations. A slightly more obscure numerical hypothesis is operative in the global arithmetic inflection formula of Corollary~\ref{thm:global_Euler_class_inflection}. Namely, linear series defined over fields of positive characteristic may be {\it inflected in every geometric point} of the underlying curve. When this happens, the linear series is of {\it (Frobenius) non-classical} type; such linear series play a central role in the approach of St\"ohr--Voloch \cite{SV} to refinements of the Hasse--Weil bound for the number of $\mb{F}_q$-rational points on curves defined over $\mb{F}_q$.\footnote{Note that a $g^r_d$ defined over a field $F$ of characteristic $p$ is never of non-classical type when $p>d$; see \cite[Cor. 1.8]{SV}.} Celebrated examples include Homma's {\it funny} plane curves defined over $\mb{F}_p$ with equations $y^d=\sum_{i=1}^{d-1} x^i z^{d-i}$ for every prime power $d=p^n$, $n \in \mb{Z}_{>0}$ \cite[Thm. 3.4]{Hom}. It would be interesting to explicitly identify all base fields $F$ not of characteristic 2 for which $|2\ell \infty_X|$ is of non-classical type on the hyperelliptic curve $y^2=f(x)$; we leave this to future work.

\section{Local arithmetic inflection formulae}\label{sec:infl}

In this section, we calculate local arithmetic inflection indices following the basic template established in Section~\ref{sec:blueprint}. In order to do so, we require one additional technical ingredient, which we now introduce.

\subsection{Hasse derivatives}\label{subsec:Hasse_Witt}
A characteristic feature of geometry in characteristic $p>0$ is the existence of nontrivial functions $x^p \in F[x]$ whose derivatives vanish to arbitrary orders. To 
adjust for this, we instead use {\it Hasse derivatives}; for the general theory, see \cite[Def. 1.1]{Vojta} (where a different nomenclature is used; our usage follows \cite{SV}). 

\medskip
Given a polynomial $f = a_n x^n + \ldots + a_1 x + a_0\in F[x]$, its {\it $k$-th Hasse derivative} is
\begin{equation}\label{hasse_derivative}
\frac{D^{k}f}{dx} := \displaystyle \mathlarger{\sum}_{i=k}^n \binom{i}{k} a_i x^{i-k} \in F[x].
\end{equation}
Note that $k! \frac{D^{k}f}{dx}$ is simply the usual $k$-th derivative of $f$; \eqref{hasse_derivative} gives an alternative way to differentiate that is well-suited to positive characteristic. There is also a  version of Taylor's formula for Hasse derivatives; namely,
\[
f = \mathlarger{\sum}_{i=0}^n \frac{D^{i}f}{dx}(a) (x-a)^i.
\]

\medskip
Hasse derivatives may be defined, more generally, for elements of an arbitrary $F$-algebra $R$. Analogues of the key properties of usual derivatives exist for their Hasse counterparts. For example, in place of the usual Leibniz rule for derivatives of products we have 
    \begin{equation} \label{eq:Hasse_Leibniz}
        \frac{D^{k}(fg)}{dz}=\mathlarger{\sum}_{i=0}^k \frac{D^{i}f}{dz}\cdot \frac{D^{k-i}g}{dz}
    \end{equation}
    and its long-form generalization
  \begin{equation} \label{eq:Hasse_Leibniz_longer}
        \frac{D^{k}}{dz}\mathlarger{\prod}_{j=1}^e f_j =\mathlarger{\sum}_{\substack{i_1+\dotsb i_e=k\\i_m \ge 0 \text{ for all }i_m}}\; \mathlarger{\prod}_{j=1}^e \frac{D^{i_j}f_j}{dz} \; .
    \end{equation}  
Similarly, 
{\it Fa\`a di Bruno's chain rule} for Hasse derivatives establishes that given 
$g \in R$ together with a choice of $f \in R[x]$, interpreted as a set-theoretic map $f: R \rightarrow R$ (which need not even be $F$-linear), we have 
\begin{equation} \label{FdB}
\frac{D^k(f \circ g)}{dz} = \mathlarger{\sum}_{\substack{\sum_{i=1}^k i c_i=k \\ c_i \geq 0 \text{ for all }i}} \binom{c_1+\cdots+c_k}{c_1,\dots,c_k} \cdot 
\left(\frac{D^{c_1+\cdots+c_k}(f)}{dx} \circ g \right) \cdot \mathlarger{\prod}_{j=1}^k \left(\frac{D^j g}{dz}\right)^{c_j}.
\end{equation}

\subsection{Hasse principal parts}\label{subsec:infl_jet} 
Next, given a line bundle $L$ on a curve $X$ and a nonnegative integer $k$, let $J^{k+1}(L)$ denote the bundle of $k$-th order {\it (Hasse) principal parts} associated with $L$. 
Its sections may be thought of $k$-th order (truncated) Taylor series expansions of sections of $L$; and it may be constructed by specifying local trivializations and transition maps between them, which glue to form $J^{k+1}(L)$. For lack of a suitable reference, we carry this out now.

\medskip
We start with local trivializations. Accordingly, 
let $U \sub X$ denote an open affine subset, 
and let $\psi_{L_U}: L_U \rightarrow \mc{O}_U$ be a local trivialization of $L$ restricted to $U$. This local trivialization singles out a distinguished basis $s:=\psi_{L_U}^{-1}(1)$. 
Shrinking $U$ if necessary, choose a local section $z \in H^0(\mc{O}_U)$ with nowhere-vanishing differential $dz \in H^0(K_U)$. The latter naturally prescribes a local trivialization $\psi_{K_U}: K_U \rightarrow \mc{O}_U$, along with a distinguished basis $dz$.\footnote{Note that the basis $dz$ need not coincide with the dual distinguished basis of $K_U$ 
induced by a local Nisnevich chart $\varphi: U \rightarrow \mb{A}^1_F$.} These local trivializations 
realize $L_U$ and $K_U$ as free $\mc{O}_U$-modules of rank 1 with specified generators $s$ and $dz$, respectively. 
Using the pair of distinguished bases $(s,dz)$, we now specify a local trivialization of $J^{k+1}(L)|_U$ via the map 
\[
    \psi^{s,dz}_{J^{k+1}(L)}: \bigoplus_{i=0}^k L_U \otimes K_U^{\otimes i} \rightarrow \mc{O}_U^{\oplus k+1}
\]
induced by direct sums and tensor powers of $\psi_{L_U}$ and $\psi_{K_U}$. The associated distinguished basis is then $(s \otimes dz^{\otimes i})_{i=0,\dotsc,k}$.

\medskip
We next describe the transition maps between local trivializations of $J^{k+1}(L)$ as above. Accordingly, 
let $(U_1,s_1,dz_1)$ and $(U_2,s_2,dz_2)$ denote triples of open sets $U_j$ together with pairs of distinguished bases of $L$ and $K_X$. We will glue $\oplus_{i=0}^k L_{U_j} \otimes K_{U_j}^{\otimes i}$ for $j=1,2$ over the intersection $W:=U_1 \cap U_2$ via a transition map $\zeta : \mc{O}_W^{\oplus k+1} \rightarrow \mc{O}_W^{\oplus k+1}$ that changes the coordinates on $(s_1 \otimes dz_1^{\otimes e})_{e=0,\dotsc,k}$ to the coordinates on $(s_2 \otimes dz_2^{\otimes e})_{e=0,\dotsc,k}$. 
To this end, shrink $U_1$ and $U_2$ (in the process shrinking $W$) so that for each $j=1,2$, the set of local sections of the form $\left( \frac{D^i f}{dz_j} \right)_{i=0,\dotsc,k}$ 
with $f \in H^0(\mc{O}_W)$ generate $\mc{O}_W^{\oplus k+1}$ as an $\mc{O}_W$-module. (To see that this is possible, 
let $p \in W$ and apply Nakayama's lemma at $p$ to the submodule of $\mc{O}_W$ generated by such sections) Now define $\zeta$ so that  $\left(\psi^{s_2,dz_2}_{J^{k+1}(U)}\right)^{-1} \circ \zeta \circ \psi^{s_1,dz_1}_{J^{k+1}(U)}$ (which is only defined over $W$) maps $\left(s_1 \otimes \frac{D^i f}{dz_1}dz_1\right)_{i=0,\dotsc,k}\in \oplus_{i=0}^k L_W \otimes K_W^{\otimes k+1}$ to $\left(s_2 \otimes \frac{D^i (fg)}{dz_2}dz_2\right)_{i=0,\dotsc,k}$. As an element of ${\rm GL}_{k+1}(\mc{O}_W)$, the first three rows of $\zeta$ are as follows: 
\[
    \begin{pmatrix}
        g & 0 & 0 & 0 & \cdots & 0\\
        \frac{D^1g}{dz_2} & g\frac{D^1z_2}{dz_1} & 0 & 0 & \cdots & 0\\
        \frac{D^2g}{dz_2} & \frac{D^1z_1}{dz_2} \frac{D^1g}{dz_1} + \frac{D^2g}{dz_1} & g\left(\frac{D^1z_1}{dz_2}\right)^2 & 0 & \cdots & 0
    \end{pmatrix}
\]
where $s_1=g \cdot s_2$. Indeed, these row vectors reflect the following consequences of the Hasse Leibniz rule \eqref{eq:Hasse_Leibniz_longer} and Fa\`a di Bruno's chain rule \eqref{FdB}:  
\begin{align*}
    fg &= f \cdot g\\
    \frac{D^1(fg)}{dz_1} &= f\cdot \frac{D^1g}{dz_2} + \frac{D^1f}{dz_1} \cdot \left( \frac{D^1z_1}{dz_2} \right) \cdot g\\
    \frac{D^2(fg)}{dz_1} &= f\cdot \frac{D^2g}{dz_2} + \frac{D^1f}{dz_1} \cdot \left( \frac{D^1z_1}{dz_2} \frac{D^1g}{dz_1} + \frac{D^2g}{dz_1} \right) +  \frac{D^2f}{dz_1} \cdot \left(\frac{D^1z_1}{dz_2}\right)^2 \cdot g
\end{align*}
To conclude that $J^{k+1}(L)$ is well-defined, it remains to explicitly check that its transition maps satisfy the cocycle condition, using \eqref{eq:Hasse_Leibniz_longer} and \eqref{FdB}; we leave this to the reader.

\medskip
The particular shape of the matrix representation of a transition map $\zeta$ leads to the following proposition:

\begin{prop}\label{prop:jet_exact_seq}
    The sequence \eqref{eq:jet_exact_seq} is short exact:
    \[
        0 \ra L \otimes K_X^{\otimes k} \ra J^{k+1}(L) \ra J^k(L) \ra 0
    \]
\end{prop}

\begin{proof}
    The fact that the $k$-th Hasse derivatives of $fg$ 
    may be expressed as a polynomial in the $e$-th derivatives of $f$ and $g$ for $e \le k$ implies that the matrix $\zeta$ is lower triangular. Moreover, 
    applying \eqref{eq:Hasse_Leibniz_longer} and \eqref{FdB} again to the $(e+1)^{\mathrm{th}}$ row of $\zeta$, we see that the corresponding diagonal entry is $g\left(\frac{D^1 z_1}{dz_2}\right)^e$, which is precisely the transition function $\zeta_e$ of $L \otimes K_X^{\otimes e}$ that identifies distinguished bases $s_1 \otimes dz_1^{\otimes e}$ and $\zeta_e \cdot (s_2 \otimes dz_2^{\otimes e})$. The exactness of \eqref{eq:jet_exact_seq} follows. 
\end{proof}

\begin{rem}\label{rmk:usual_jet}
    Recall that the $e$-th Hasse derivative $\frac{D^ef}{dx}$ is $\frac{1}{e!}\frac{d^ef}{dx}$ when the characteristic of $F$ is either zero or larger than $e$. Consequently, replacing every instance of Hasse derivative that appears in $\zeta$ by the usual derivative has the effect of rescaling the $(e+1)$-th row of $\zeta$ by $e!$; in particular, the factor $e!$ is independent of the choice of local trivializations of $J^{k+1}(L)$. This new matrix $\zeta^{\pr}$ coincides with the transition map for the usual bundle of $k$-th order principal parts of \cite[Exp. II, App. II.1.2.4]{SGA6}. 
\end{rem}

\subsection{Hasse Wronskians revisited}\label{subsec:Hasse_Wronskian}

Let $(L,V)$ be a $g^r_d$ on a curve $X$ as in \S \ref{subsec:lin_series_hypell}. We first describe the evaluation map \eqref{evaluation_morphism} via the local trivializations given in \S \ref{subsec:infl_jet}. These enable us 
to describe $\mathrm{ev}(v)$ as a global section of $J^{r+1}(L)$, which we then $\mc{O}_X$-linearly extend to realize $\mathrm{ev}$ as a global section of $\mathcal{H}\mathrm{om}(V \otimes \mc{O}_X, J^{r+1}(L))$.

\medskip
Now choose any section $v \in V \subset H^0(L)$, and choose any triple $(U,s,dz)$ as in \S \ref{subsec:infl_jet} that represents a local trivialization $\psi^{s,dz}_{J^{r+1}(L)}:J^{r+1}(L)|_U \rightarrow \mc{O}_U^{\oplus r+1}$. Set $f=\psi_{L_U}(v)$; this means that $v=fs$ in $H^0(L|_U)$. We now define $\mathrm{ev}(v)$ on $U$ via $\psi^{s,dz}_{J^{r+1}(L)}(\mathrm{ev}(v)) = \left( \frac{D^if}{dz} \right)_{i=0,\dotsc,r}$. 
Using the 
transition maps $\zeta$ of \S \ref{subsec:infl_jet}, 
we glue the local descriptions of $\mathrm{ev}(v)$ to produce a global section of $J^{r+1}(L)$. We finally extend $\mathrm{ev}(v)$ $\mc{O}_X$-linearly for every $v \in V$ to obtain $\mathrm{ev} \in \mathrm{Hom}(V \otimes \mc{O}_X, J^{r+1}(L))$. 

\medskip
Recall from \S \ref{subsec:lin_series_hypell} that associated to every local trivialization $\psi_\la:V \otimes \mc{O}_X \rightarrow \mc{O}_X^{r+1}$ with a chosen distinguished basis $\la$ is a homomorphism $W(\la)=(\psi^*_{\la})^{-1}(\mathrm{ev}) \in \mathrm{Hom}(\mc{O}^{r+1},J^{r+1}(L))$. Using the local descriptions of $\mathrm{ev}$ as above, we represent $W(\la)$ as a matrix in the coordinates $(U,s,dz)$; namely,
\begin{equation}\label{wronskian_matrix}
    \psi^{(s,dz)}_{J^{r+1}(L)} \circ W(\la)_U= \begin{pmatrix} \lambda_{0,s} & \cdots & \lambda_{r,s} \\ \dfrac{D^{1}\lambda_{0,s}}{dz} & \cdots & \dfrac{D^{1}\lambda_{r,s}}{dz} \\ \vdots & \ddots & \vdots \\ \dfrac{D^{r}\lambda_{0,s}}{dz} & \cdots & \dfrac{D^{r}\lambda_{r,s}}{dz} \end{pmatrix} \in M_{r+1}(\mc{O}_{U})
\end{equation}
where $\lambda_{i,s}:=\psi_{L_U}(\la_i)$. Taking the determinant of \eqref{wronskian_matrix} and gluing local trivializations together, we obtain the {\it Hasse Wronskian} $w(\la)$.

\medskip
We return now to the case of the complete linear series associated to  $L=\mc{O}_X(2\ell \infty_X)$ on a hyperelliptic curve $X$. In order to calculate the local Euler index $\mathrm{ind}_p w(\la)$ in a point $p$ of $X$, we apply the strategy outlined in \S \ref{subsec:loc_Euler_general}; in so doing, we replace $w(\la)$ by a local section defined along an affine line. Accordingly, choose a local Nisnevich chart $\varphi_\ga:U_\ga \rightarrow \mathrm{Spec}\: F[\ga]$ with adapted local trivializations $\psi_{L,\ga}$ and $\psi_{T_{U_\ga}}$ of $L$ and $T_X$, respectively, as in \S \ref{subsec:nisnevich_coords_and_projections}; here $\ga$ is either a $F$-linear combination of $x$ and $y$ or $w$, and $U_\ga$ contains $p$. By choosing an appropriate distinguished basis $s$ of $L|_{U_\ga}$, the top wedge of $\psi_{J^{r+1}(L)}^{(s,d\ga)}$ coincides with the local trivialization $\psi_{\mc{E},\ga}$ as in \S \ref{subsec:nisnevich_coords_and_projections}, where $\mc{E}:=\det J^{r+1}(L)$. Observing that $\ga$, up to $F$-translation, is a uniformizer for $\bar p:=\varphi_\ga(p)$\footnote{When $k(x)/F$ is a nontrivial extension, this notation is abusive and a base change is required; see Remark~\ref{rmk:geom_pt}.}, we 
now systematically replace each entry of \eqref{wronskian_matrix} by an appropriate approximation in $F[\ga]$.\footnote{More precisely, the error of approximation lies in $\mf{m}^{2n}$, where $\mf{m}$ is the maximal ideal of $p$, $Z:=(w(\la)=0)$, and $\mf{m}^n=0$ as an ideal of $\mc{O}_{Z,p}$.} Call this new matrix $W_\ga(\la) \in M_{r+1}(F[\ga])$, and let $w_\ga(\la)$ denote its determinant. We now have $\mathrm{ind}_p w(\la)=\mathrm{ind}_{\bar p}w_\ga(\la)$, which may be computed using Lemmas \ref{lem:locEuler} and \ref{lem:locEuler_bch}. 
The vanishing locus of $w(\la)$ defines the {\it Hasse inflection locus}, which agrees with the usual inflection locus in characteristic $0$.

\subsection{Local arithmetic Euler indices along $R_{\pi}$}\label{subsec:loc_ind_gen_ramif}
We will argue separately according to whether the underlying ramification point is $\infty_X$; and whether $\ell \leq g$ or $\ell > g$, which conditions the shape of the monomial basis \eqref{eq:monomial_basis}). 

\subsubsection{Local arithmetic Euler indices along $R_{\pi} \setminus \{\infty_X\}$}\label{subsubsec:loc_ind_gen_ramif_affine}
We begin by proving the Hasse-inflectionary analogue of \cite[Theorem 3.3]{BCG}. 

\begin{thm}\label{thm:infl_conc_ramif_1}
    Let $X$ denote a hyperelliptic curve defined over a field $F$ of characteristic $p \neq 2$ as above. Assume that $\ell \leq g$, in which case the complete linear series defined by $\mc{O}(2\ell \infty_X)$ has basis $\lambda=(1,x,x^2,\dotsc,x^{\ell})$. The zero locus of the Hasse Wronskian $w(\lambda)$ is $\binom{\ell+1}{2}$ times the ramification divisor $R_{\pi}$. Moreover, near any inflection point $p \neq \infty_X$ in $R_{\pi}$, we have 
    \[
    w(\la)=\left(\frac{Dx}{dz}\right)^{\binom{\ell+1}{2}}
    \]
    where $z$ is a local Nisnevich uniformizer in $p$.
\end{thm}
\begin{proof}
    We proceed as in \cite[Theorem 3.3]{BCG}, 
    using the Hasse Leibniz rule \eqref{eq:Hasse_Leibniz_longer} in place of the usual Leibniz rule. More precisely, we will show that the square matrix $W(\lambda)$ may be column--reduced to a lower triangular matrix for which the diagonal entry of the column indexed by each $\lambda_{i}(z)$ is $\left(\frac{Dx}{dz}\right)^i$; the theorem follows from the fact that away from $\infty$, $\frac{Dx}{dz}$ vanishes precisely to order 1 along the ramification divisor $R_{\pi}$.
    
    
    \medskip
   To this end, note that the first column of $W(\lambda)$, corresponding to $\lambda_{0}=1$, is zero in every entry except the first one, which is 1. Similarly, the second column, corresponding to $\lambda_{1}=x$, has entries $\frac{D^{k}x}{dz}$, $k \ge 0$. On the other hand, whenever $e \geq 2$, we may rewrite the entries of each column corresponding to $\lambda_{e}$ 
    using equation~\eqref{eq:Hasse_Leibniz_longer}, obtaining
    \begin{equation}\label{eq:Hasse_Leibniz_power}
        \frac{D^{k}(x^e)}{dz}=\mathlarger{\sum}_{\gamma=1}^e\binom{e}{\gamma}x^{e-\gamma}\mathlarger{\sum}_{\substack{i_1+\dotsb i_\gamma=k\\i_m \ge 1 \text{ for all }i_m}}\; \mathlarger{\prod}_{j=1}^{\gamma} \frac{D^{i_j}x}{dz}
    \end{equation}
    for every $k \geq 1$. Now define $C_{0}$ to be a column vector all of whose entries are zero except for the first entry $C_{0,0}=1$, and for every $\gamma \ge 1$, let $C_{\gamma}$ denote a column vector with entries
    \[
        C_{k,\gamma}=\mathlarger{\sum}_{\substack{i_1+\dotsb i_\gamma=k\\i_m \ge 1 \; \text{ for all }i_m}}\; \mathlarger{\prod}_{j=1}^{\gamma} \frac{D^{i_j}x}{dz} 
    \]
    for $k=0,\dotsc,\ell$. Note that each column $W(\la)_{e}$ of $W(\la)$ corresponding to the basis element $\lambda_{e}=x^e$ for $e \ge 1$ is spanned by linearly independent vectors $C_{\gamma}$ for $\gamma=0,\dotsc,e$. More precisely, we have
    \[
        W(\la)_e=\mathlarger{\sum}_{\gamma=0}^e\binom{e}{\gamma}x^{e-\gamma}C_{\gamma} \; .
    \]
    Using this, we immediately conclude that the matrix corresponding to $W(\lambda)$ column-reduces to a matrix with columns $C_{\gamma}$, $0 \le \gamma \le \ell$.
    
    {\fl On the other hand, note} that $C_{k,\gamma}=0$ when $1 \le k <\gamma$, as not all $i_{m}$, $m=1,\dotsc,\gamma$, satisfy $i_m \geq 1$. Since
    \[
        C_{\gamma,\gamma}=\mathlarger{\sum}_{\substack{i_1+\dotsb i_\gamma=\gamma\\i_m \ge 1 \text{ for all }i_m}}\; \mathlarger{\prod}_{j=1}^{\gamma} \frac{D^{i_j}x}{dz} =\mathlarger{\prod}_{j=1}^{\gamma}\frac{Dx}{dz}
    \]
    it follows that the diagonal entries of our column-reduced matrix are $(\frac{Dx}{dz})^{m}$, $0 \le m \le \ell$. Multiplying them together yields the desired local description of $w(\lambda)$.
\end{proof}

\begin{thm}\label{thm:euler_index_ramif_l_leq_g}
Let $X$ denote a hyperelliptic curve defined over a field $F$ of characteristic $p \neq 2$ as above. Whenever $\ell \leq g$ and $\ell \equiv 1 \pmod{4}$, the local Euler index of the complete linear series $|2\ell \infty_X|$ in ${\rm GW}(F)$ associated to a ramification point $(\ga,0) \in R_{\pi} \setminus \{\infty_X\}$ of the hyperelliptic projection $\pi: X \ra \mb{P}^1$ is given by
\begin{equation}\label{eq:Euler_ramif_l<=g}
    \mathrm{ind}_{(\gamma,0)}w(\lambda)=\mathrm{Tr}_{k(\gamma)/F}\left(\frac{\binom{\ell+1}{2}-1}{2}\cdot \mathbb{H}+\left<\frac{(D^1f)(\gamma)}{2}\right>\right).
\end{equation}
\end{thm}

\begin{proof}
To compute the local Euler index at a ramification point $(\gamma,0) \in R_{\pi}$ using Theorem~\ref{thm:infl_conc_ramif_1}, we first need to understand the local description of $w(\lambda)$ when $z=u_b=y-bx$. In this case, the affine equation $y^2=f(x)$ for $X$ becomes $(u_b+bx)^2=f(x)$, and differentiating both sides of the latter equation with respect to $u_b$ yields
\begin{align*}
    2(u_b+bx)\left(1+b\frac{D^1x}{du_b}\right)&=(D^1f)(x)\cdot \frac{D^1x}{du_b}.
\end{align*}
After collecting terms involving $\frac{D^1x}{du_b}$, we obtain
\begin{equation}\label{eq:derivative_of_x}
    \frac{D^1x}{du_b}=\frac{2(u_b+bx)}{(D^1f)(x)-2(u_b+bx)b}.
\end{equation}
Here $x(u_b)=\gamma$ when $u_b=-b\gamma$. Note that $\frac{D^1x}{du_b}$ vanishes at $u_b=-b\gamma$ to order 1, in light of the facts that $f(\ga)=0$ and $D^1f(\gamma) \neq 0$, and that the power series expansion of $\frac{D^1x}{du_b}$ is of the form
\begin{equation}\label{eq:power_series_first_derivative}
    \frac{D^1x}{du_b}=\frac{2(u_b+b\gamma)}{(D^1f)(\gamma)}+h(u_b)(u_b+b\gamma)^2
\end{equation}
where $h \in F\llbracket u_b\rrbracket$. Applying Lemma~\ref{lem:locEuler} and \ref{lem:locEuler_bch} in tandem with Theorem~\ref{thm:infl_conc_ramif_1}, \eqref{eq:Euler_ramif_l<=g} follows.
\end{proof}

\medskip
\noindent Now suppose that $\ell>g$. It is difficult to directly compute $w(\lambda)$ with respect to a local Nisnevich coordinate. To circumvent this difficulty, we first identify the lowest-order terms of $w(\lambda)$ with respect to the \'etale coordinate $y$. In formulating this result, which refines \cite[Theorem 5.7]{BCG}, we let $M(\ell,g)$ denote the $(g+1) \times (g+1)$ matrix with entries $M_{ij}= \binom{\ell-g+j}{2j-i}$, $0 \leq i,j \leq g$. 

 \begin{thm}\label{thm:infl_conc_ramif}
    Assume that $\ell > g$, and let $\la:=(1,y,\dotsc,x^{\ell-g-1},x^{\ell-g-1}y;x^{\ell-g},x^{\ell-g+1},\dotsc,x^\ell)$ denote the corresponding monomial basis of the complete linear series $|\mc{O}(2\ell \infty_X)|$, where $\mathrm{char}(F) \neq 2$ and $\det M(\ell,g)$ is nonzero in $F$. 
    With respect to the local \'etale coordinate $y$, the lowest-order term of the Hasse Wronskian $w(\lambda)$ in a ramification point $(\ga,0) \in R_{\pi} \setminus \{\infty_X\}$ is that of
    \[
    \det M(\ell,g) \cdot \left(D^1_y x\right)^{\binom{g+1}{2}} (D^2_y x)^{\ell(\ell-g)}
    \]
    where $D^i_yx= \frac{D^i x}{dy}$. 
\end{thm}

\begin{rem}\label{rem:Gessel--Viennot}
Let $\ell_1$ and $\ell_2$ denote the lines $x+y=0$ and $2y+x=2\ell-2g$ in the $xy$-plane, respectively. Let $a_i$ (resp., $b_i$) denote the point of $\ell_1$ (resp., $\ell_2$) with coordinates $(i,-i)$ (resp., $(2i,\ell-g-i)$), $i=0,\dots,g$.
The Gessel--Viennot lemma \cite{GV} implies that (the integer underlying) $\det M(\ell,g)$ is equal to the number of non-intersecting lattice paths connecting the $(g+1)$-tuple of points $a_i$, $i=0,\dots,g$ with the $(g+1)$-tuple of points $b_i$, $i=0,\dots,g$.
\end{rem}

\begin{proof}
An outline of the proof is as follows. We begin by deriving an {\it inflectionary basis} $\la^{\ga}$ adapted to the ramification point $(\ga,0)$ from our usual basis $\la$ of global sections. This inflectionary basis, in turn, naturally singles out a row-reduction of the corresponding Hasse Wronskian matrix $W(\la^{\ga})$. Every entry of the row-reduced version of $W(\la^{\ga})$ is a sum of monomials in Hasse derivatives $D^j_y x$ of $x$ with respect to $y$; and its contribution to the lowest order terms of its determinant is itself a matrix of monomials in $D^1_y x$ and $D^2_y x$ whose associated matrix of coefficients is $M(\ell,g)$. 

\medskip
More explicitly now, note that in each ramification point $(\ga,0)$, 
\begin{equation}\label{inflectionary_basis}
\la^{\ga}:=(1, y, (x-\ga), (x-\ga)y, \dots, (x-\ga)^{\ell-g-1}, (x-\ga)^{\ell-g-1}y; (x-\ga)^{\ell-g}, (x-\ga)^{\ell-g+1}, \dots, (x-\ga)^{\ell})
\end{equation}
determines an inflectionary basis for $\mc{O}(2\ell \infty_X)$, i.e., the elements of $\la^{\ga}$ are ordered according to their orders-of-vanishing in the uniformizer $y$, which are strictly increasing. Note that for each $i$, the $i^{\mathrm{th}}$ element of $\la^{\ga}$ is a linear combination of the first $i$ elements of $\la$ in which the $i^{\mathrm{th}}$ element of $\la$ appears with coefficient one; as a result, the Hasse Wronskians of $\la$ and $\la^{\ga}$ are equal. 

\medskip
As a result, we may assume without loss of generality that $\ga=0$. 
Note first that the Hasse Leibniz rule \eqref{eq:Hasse_Leibniz} implies
\[
D^k_y (x^j y)= D^{k-1}_y x^j+ y D^k_y x^j
\]
for all positive integers $j$ and $k$; it follows easily that the column of $W(\la)$ indexed by $x^j y$ may be replaced by a column whose $k$th entry (when counting starting from 0 at the top of the column) is $D^{k-1}_y x^j$, for every $j=0,\dots,\ell-g-1$. 
Computing the Hasse derivative of powers of $x$ with respect to $y$ via Fa\`a di Bruno's chain rule for Hasse derivatives \eqref{FdB}, we find that
\begin{equation}\label{derivative_of_x^j}
D^k_y x^j= \mathlarger{\sum}_{\substack{\sum_{i=1}^k i c_i=k\\ c_i \geq 0 \text{ for all }i}} \binom{c_1+\cdots+c_k}{c_1,\dots,c_k} 
\binom{j}{c_1+\dots+c_k} x^{j-(c_1+\dots+c_k)} \cdot \mathlarger{\prod}_{i=1}^k (D^i_y x)^{c_i}.
\end{equation}
We claim that via further column reduction, all entries of the form $D^k_y x^j$ may be systematically replaced by
\begin{equation}\label{pure_derivative_columns}
\mathlarger{\sum}_{\substack{\sum_{i=1}^k i c_i=k \\ \sum_{i=1}^k c_i=j}} \binom{j}{c_1,\dots,c_k}
\mathlarger{\prod}_{i=1}^k (D^i_y x)^{c_i}
\end{equation}
or equivalently, that all terms in \eqref{derivative_of_x^j} indexed by $k$-tuples of non-negative integers $(c_1,\dots,c_k)$ with $\sum_{i=1}^k c_i<j$ may be ignored. Notice that this is vacuously true when $j=0,1$. To see this when $j>1$, fix a strictly positive integer $j^{\pr}<j$, and let $C_j$ (resp., $C_j^{\pr}$) denote any column whose $k$th entry is $D^k_y x^j$ (resp., $D^{k-1}_y x^j$); note that with the exception of the ``trivial" columns indexed by 1 and $y$, every column of $W(\la)$ is either of the form $C_j$ or $C_j^{\pr}$ for some $j$.

\medskip
We now iteratively define a sequence of column reductions, as follows. Let $C_{j;1}:= C_j-jx^{j-1} C_1$ and $C^{\pr}_{j;1}:= C^{\pr}_j-jx^{j-1} C^{\pr}_1$ for every $j=2,\dots,\ell$. Replace column $C_j$ (resp., $C_j^{\pr}$) by $C_{j;1}$ (resp., $C^{\pr}_{j;1}$) for every $j=2,\dots,\ell$; doing so eliminates all terms involving $x^{j-1}$. For convenience, continue labeling the columns of $W(\la)$ by $C_j$ and $C_j^{\pr}$, respectively. Now replace $C_j$ (resp., $C_j^{\pr}$) by $C_{j;2}:=C_j-\binom{j}{2}x^{j-2} C_2$ (resp., $C^{\pr}_{j;2}:=C^{\pr}_j-\binom{j}{2}x^{j-2} C^{\pr}_2$) for every $j=3,\dots, \ell$; doing so eliminates all terms involving $x^{j-2}$. Iterating this procedure, at the $m^\mathrm{th}$ step replace $C_j$ (resp., $C^{\pr}_j$) by $C_{j;m}:= C_j-\binom{j}{m}x^{j-m} C_m$ (resp., $C^{\pr}_{j;m}:= C^{\pr}_j-\binom{j}{m}x^{j-m} C^{\pr}_m$). After iterating $\ell$ times, the columns of $W(\la)$ are clearly as in \eqref{pure_derivative_columns}.

\medskip
The preceding analysis shows, in particular, that the $k$th entry of the column indexed by $x^{\ell-g+j}$, $j=0,\dots,g$ of (the column-reduced version of) $W(\la)$ is a sum of monomials $\prod_{i=1}^{2\ell-g} (D^i_y x)^{c_i}$ whose associated (appropriately reordered) exponent vector $(c_1,\dots,c_{2\ell-g})$ determines a partition $(i^{c_i})_i$ of weight $k$ with $\ell-g+j$ parts. Here we may assume $\ell-g+j \leq k \leq \ell$ without loss of generality; indeed, whenever $k<\ell-g+j$, there are no such partitions, and the corresponding entry of $W(\la)$ is zero.
Likewise, it is easy to see that none of the leftmost $(2\ell-2g)$ columns of $W(\la)$ are divisible by $D^1_y x$, but that their reductions modulo $D^1_y x$ are zero above the diagonal. Indeed, this follows from the fact that the only weight-$k$ partitions with $j$ parts for which $j \leq k \leq 2j-1$ have at least one singleton part.

\medskip
Note that modulo $D^1_y x$, the diagonal entries that appear among the leftmost $(2\ell-2g)$ columns of $W(\la)$ are precisely $1, 1, (D^2_y x), (D^2_y x), \dots, (D^2_y x)^{\ell-g-1}, (D^2_y x)^{\ell-g-1}$; their product is $(D^2_y x)^{(\ell-g)(\ell-g-1)}$. In order to conclude, we will more closely examine the rightmost $(g+1)$ columns of $W(\la)$, namely those indexed by $x^{\ell-g+i}$, $i=0,\dots,g$. Accordingly, let $M$ denote the $(g+1) \times (g+1)$ matrix obtained from the submatrix of $W(\la)$ determined by the bottom-most $(g+1)$ rows of the rightmost $(g+1)$ columns, recording the multiplicity with which $D^1_y x$ divides the corresponding entry of $W(\la)$, i.e., its ``$D^1_y x$-adic valuation''; and as a matter of convention consistent with this valuative perspective, we take this multiplicity to be $\infty$ whenever the corresponding entry of $W(\la)$ is zero. Let $M_i$, $i=0,\dots,g$ denote the $i$th column of $M$. Whenever it is finite, the $k$th entry of $M_i$ is precisely the minimal number of singleton parts of weight-$(2\ell-2g+k)$ partitions with $\ell-g+i$ parts, $k=0,\dots,g$. This minimal number, in turn, depends on how large the number of parts is relative to the weight. More precisely, whenever $2\ell-2g+k>2(\ell-g+i)$, i.e., when $k>2i$, the minimal number is clearly zero. On the other hand, whenever $k \leq 2i$, the minimal number is realized by the partition $(2^{\la_1},1^{\la_2})$ for which $\la_1+\la_2=\ell-g+i$ and $2\la_1+\la_2=2\ell-2g+k$; solving, we find $\la_1=\ell-g+k-i$ and $\la_2=2i-k$. Since we require $\la_1$ to be non-negative, it follows that whenever $k \leq 2i$, the $k$th entry of $M_i$ is either $2i-k$, with $i \leq \ell-g+k$; otherwise, it is $\infty$.

\medskip
To compute the multiplicity with which $D^1_y x$ divides $w(\la)$, the key point is that the tropical {\it permanent} ${\rm perm}(M)$ of $M$ with respect to addition and taking minima is precisely $\binom{g+1}{2}$.\footnote{That is, we compute the permanent of the matrix $M$ with entries in $\ov{\mb{N}}$ using the tropical operations $+$ and $\min$ in place of the usual operations of $\times$ and $+$, respectively; the tropical neutral element is $\infty$ instead of $0$.} Indeed, the tropical permanent of $M$ is bounded below by the tropical permanent of the matrix $M^{\pr}$ whose $(k,i)^{\mathrm{th}}$ entry is $2i-k$, for every $0 \leq i,k \leq g$. On the other hand, every $(g+1)$-tuple tropical product that contributes to ${\rm perm}(M^{\pr})$ is the same, namely $\binom{g+1}{2}$, which also gives the contribution of the diagonal to ${\rm perm}(M)$. 
Since $(k,i)^{\mathrm{th}}$ entry of $M'$ records $\lambda_2$ of the corresponding partition $(2^{\lambda_1},1^{\lambda_2})$ (where $\lambda_1,\lambda_2$ are allowed to be negative) of $2\ell-2g+k$ by $\ell-g+i$ parts, the determinant of the bottom-right $(g+1) \times (g+1)$ part of $W(\lambda)$ must be a constant times $(D^1_yx)^{\binom{g+1}{2}}(D^2_yx)^{(g+1)(\ell-g)}$. So we deduce that, up to a scalar multiple, the lowest-order term of $w(\la)$ is equal to $(D^1_y x)^{\binom{g+1}{2}} (D^2_y x)^{\ell(\ell-g)}$.

\medskip
More precisely, the lowest-order term of $w(\la)$ is equal to $(D^1_y x)^{\binom{g+1}{2}} (D^2_y x)^{\ell(\ell-g)}$ times a $(g+1) \times (g+1)$ determinant of Fa\`a di Bruno coefficients associated with monomials $(D^1_y x)^{c_1}(D^2_y x)^{c_2}$ in the (entries of) the lower right-hand $(g+1) \times (g+1)$ submatrix of $W_{\al}(\la)$ that contribute to the lowest-order term of $w(\la)$. According to the preceding two paragraphs, those monomials $(D^1_y x)^{c_1}(D^2_y x)^{c_2}$ in the $(k,i)$ entry of the lower right-hand $(g+1) \times (g+1)$ submatrix of $W(\la)$ that contribute are those for which $c_2=\ell-g+k-i$ and $c_1=2i-k$, respectively, whenever these quantities are nonnegative (if either $\ell-g+k-i$ or $2i-k$ is negative, there is no contributing monomial, and the associated coefficient is zero). According to \eqref{pure_derivative_columns}, this is precisely the $(k,i)$-entry of the matrix $M(\ell,g)$; the desired result follows.
\end{proof}

With Theorem~\ref{thm:infl_conc_ramif} in hand, we can compute the local Euler index of $|\mc{O}(2\ell \infty_X)|$ in a hyperelliptic ramification point whenever $\ell>g$.
\begin{thm}\label{thm:euler_index_ramif_l>g}
Let $X$ denote a hyperelliptic curve defined over a perfect field $F$ of characteristic not equal to 2 as above. Whenever $\ell>g$, $2\ell - g \equiv 1 \pmod{4}$, and $\det M(\ell,g)$ is nonzero in $F$, the local Euler index of the complete linear series $|2\ell \infty_X|$ in ${\rm GW}(F)$ associated to a ramification point $(\ga,0) \in R_{\pi} \setminus \{\infty_X\}$ of the hyperelliptic projection $\pi: X \ra \mb{P}^1$ is given by
\begin{align}
    &\mathrm{ind}_{(\gamma,0)}w(\lambda) \notag \\
    &=\begin{cases}
            \mathrm{Tr}_{k(\gamma)/F}\left(\dfrac{1}{2}\binom{g+1}{2}\cdot \mb{H}\right) & \text{if $\ell$ is even and $g \equiv 3 \pmod{4}$}\\
            \noalign{\vskip9pt}
            \mathrm{Tr}_{k(\gamma)/F}\left(\dfrac{\binom{g+1}{2}-1}{2}\cdot \mb{H} +\left< 2(\det M(\ell,g)) (D^1f)(\gamma) \right>\right) & \text{if $\ell$ is odd and $g \equiv 1 \pmod{4}$}.
        \end{cases} \label{eq:Euler_ramif_l>g}
\end{align}
\end{thm}

\begin{proof}
To compute the local Euler index at a ramification point $(\gamma,0) \in R_{\pi}$ using Theorem~\ref{thm:infl_conc_ramif}, we first need to rewrite $w(\lambda)$ in terms of the local Nisnevich coordinate $u_b:=y-bx$. 
Note that equation \eqref{eq:derivative_of_x} describes $D^1_{u_b}x$, and Fa\`a di Bruno's chain rule \eqref{FdB} implies that $D^1_y x=(D^1_{u_b}x)(D^1_{u_b}y)^{-1}$ (where we interpret $u_b$ as a function of $y$). Since $y = u_b + bx$, we have $(D^1_{u_b}y)(-b\gamma) = 1$ as $(D^1_{u_b}x)(-b\gamma)=0$ by \eqref{eq:power_series_first_derivative}. Thus, $D^1_y x$ vanishes once at $-b\gamma$, as a function of $u_b$.

\medskip
Similarly, in order to rewrite $D^2_yx$ as a function of $u_b$, we first apply Fa\`a di Bruno, which yields
\begin{align*}
    D^2_yx=(D^1_{u_b}x)(D^2_yu_b)+(D^2_{u_b}x)(D^1_yu_b)^2.
\end{align*}
To understand the implications of this, we first rewrite $D^2_{u_b}x$. Applying Fa\`a di Bruno to compute the second Hasse derivative of $(u_b+bx)^2=f(x)$, we obtain
\begin{align*}
\left(1+bD^1_{u_b}x \right)^2+2(u_b+bx)b D^2_{u_b}x = (D^2f)(x)\left(D^1_{u_b} x \right)^2+(D^1f)(x)D^2_{u_b}x.
\end{align*}
After collecting terms involving $D^1_{u_b}x$, we obtain
\[
    D^2_{u_b}x= \frac{(b^2-(D^2f)(x))\left(D^1_{u_b}x\right)^2+ 2b D^1_{u_b}x+1}{(D^1f)(x)-2(u_b+bx)b} \label{eq:second_derivative_of_x}.
\]
Note that when $u_b=-b\gamma$, we have $x(-b\gamma)=\gamma$ and $D^1f(\gamma) \neq 0$. Moreover, the power series expansion of $D^1_{u_b}x$ of equation \eqref{eq:power_series_first_derivative} implies that the power series expansion of $D^2_{u_b}x$ is of the form
\begin{equation*}
D^2_{u_b}x=\frac{1}{(D^1f)(\gamma)}+q(u_b)(u_b+b\gamma)
\end{equation*}
where $q \in F\llbracket u_b\rrbracket$. According to Lemma~\ref{lem:locEuler}, to compute the local Euler index we need only the lowest order term of the power series expansion of $D^2_yx$ at $u_b=-b\gamma$; so we may ignore $D^2_yu_b$, as $D^2_{u_b}x$ and $D^1_yu_b$ do not vanish at $u_b=-b\gamma$ but $D^1_{u_b}x$ does.


\medskip
{\fl Putting} this all together, we see that as a function of $u_b$, the lowest order term of $w(\lambda)$ with respect to $y$ is
\[
    (D^1_{u_b}y)^{\binom{2\ell-g+1}{2}}\det M(\ell,g) \cdot (D^1_{u_b}x)^{\binom{g+1}{2}}(D^1_{u_b}y)^{-\binom{g+1}{2}}((D^1_{u_b}x)(D^2_yu_b)+(D^2_{u_b}x)(D^1_yu_b)^2)^{\ell(\ell-g)},
\]
where $(D^1_{u_b}y)^{\binom{2\ell-g+1}{2}}$ comes from the transition map described in \S \ref{subsec:nisnevich_coords_and_projections}.
Simplifying and ignoring terms with $D^2_yu_b$, our local Euler index reduces to that of
\[
    \det M(\ell,g) \cdot (D^1_{u_b}y)^{\binom{2\ell-g+1}{2}-\binom{g+1}{2}-2\ell(\ell-g)} (D^1_{u_b}x)^{\binom{g+1}{2}}(D^2_{u_b}x)^{\ell(\ell-g)}.
\]
Applying Lemma~\ref{lem:locEuler} and \ref{lem:locEuler_bch} in tandem with Theorem~\ref{thm:infl_conc_ramif}, equation~\eqref{eq:Euler_ramif_l>g} follows.
\end{proof}

\subsubsection{Local arithmetic Euler indices at $\infty_X$}\label{subsubsec:loc_ind_gen_ramif_infty}
Recall from Section~\ref{subsec:toric_geometry} that along $U_{\tau_2} \cong \mb{A}^2_{z,w}$, $X$ is defined by $w^2 = h(z)$, where $h(z)= z^{2g+2}f(z^{-1})$. According to Section~\ref{subsec:nisnevich_coords_and_projections}, the transition map for $L$ that results from exchanging any local coordinate away from $\infty$ for the local Nisnevich coordinate $w$ at $\infty$ is multiplication by $z^{\ell}$.

\medskip
Whenever $\ell \le g$, the basis $\lambda = (1,x,x^2,\dotsc,x^{\ell})$ for $|2\infty_X|$, when 
rewritten in terms of $z$ and $w$, becomes
$(z^{\ell},z^{\ell-1},z^{\ell-2},\dotsc,1)$. The latter is a permutation of $(1,z,z^2,\dotsc,z^{\ell})$ with sign equal to $(-1)^{\lfloor \frac{\ell+1}{2} \rfloor} = -1$, since we assume $\ell \equiv 1 \pmod{4}$. The induced action of the permutation on Wronskian matrices yields $w(\lambda) = (-1)^{\binom{\ell+1}{2}} \cdot w(1,z,z^2,\dotsc,z^{\ell})$. As a result, we easily obtain analogues of Theorems~\ref{thm:infl_conc_ramif_1} and \ref{thm:euler_index_ramif_l_leq_g} by replacing $f$ by $h$, $y$ by $t$ and $u_b = y - bx$ by $w$, 
$\gamma$ by $0$, and $x$ by $z$.

\begin{thm}\label{thm:infl_conc_ramif_infty}
    Assume that $\mathrm{char}(F) \neq 2$ and $\ell \le g$, in which case the complete linear series defined by $\mc{O}(2\ell \infty_X)$ has basis $\lambda = (z^{\ell},z^{\ell-1},\dotsc,z,1)$. In terms of the local Nisnevich coordinate $w$ at $\infty_X$, we have
    \[
        w(\lambda)=-\left(\frac{Dz}{dw}\right)^{\binom{\ell+1}{2}}.
    \]
\end{thm}

\begin{thm}\label{thm:euler_index_ramif_l_leq_g_infty}
    Let $X$ denote a hyperelliptic curve defined over a field $F$ of characteristic $p \neq 2$ as above. Whenever $\ell \le g$ and $\ell \equiv 1 \pmod{4}$, the local Euler index of the complete linear series $|2\ell \infty_X |$ in ${\rm GW}(F)$ associated to $\infty_X$ is given by
    \begin{equation}\label{eq:Euler_ramif_l<=g_infty}
        \mbox{\rm{ind}}_{\infty_X}w(\lambda) = \left( \frac{\binom{\ell+1}{2}-1}{2} \cdot \mb{H} + \left< -\frac{(D^1h)(0)}{2} \right> \right).
    \end{equation}
\end{thm}

\medskip
\noindent Now suppose that $\ell > g$. Rewritten in terms of $z$ and $w$, the basis 
\[\lambda = (1,y,\dotsc,x^{\ell-g-1},x^{\ell-g-1}y;x^{\ell-g},x^{\ell-g+1},\dotsc,x^{\ell})\] 
becomes
\[
(z^{\ell},z^{\ell-g-1}w,z^{\ell-1},z^{\ell-g-2}w,\dotsc,z^{g+1},w;z^g,z^{g-1},\dotsc,1).
\]
The latter is a permutation of $(1,w,z,zw,\dotsc,z^{\ell-g-1},z^{\ell-g-1}w;z^{\ell-g},z^{\ell-g+1},\dotsc,z^{\ell})$ with sign 
\[
(-1)^{\lfloor \frac{\ell+1}{2} \rfloor} \cdot (-1)^{\lfloor \frac{\ell-g}{2} \rfloor} = (-1)^{1+\binom{\ell-g}{2}}
\]
as $(-1)^{\lfloor \frac{n}{2} \rfloor} = (-1)^{\binom{n}{2}}$ for every $n \in \mb{N}$. Essentially the same argument as that used when $\ell \le g$ yields the following analogues of Theorems~\ref{thm:infl_conc_ramif} and \ref{thm:euler_index_ramif_l>g}.

\begin{thm}\label{thm:infl_conc_ramif_infty_2}
    Let $\la:=(z^{\ell},z^{\ell-g-1}w,z^{\ell-1},z^{\ell-g-2}w,\dotsc,z^{g+1},w;z^g,z^{g-1},\dotsc,1)$ denote the monomial basis of the complete linear series $|\mc{O}(2\ell \infty_X)|$ as above when $\mathrm{char}(F) \neq 2$, $\ell > g$, and $\det M(\ell,g)$ is nonzero in $F$. 
    With respect to the local \'etale coordinate $w$, the lowest-order term of the Hasse Wronskian $w(\lambda)$ at $\infty_X$ is equal to that of
    \[
    (-1)^{1+\binom{\ell-g}{2}} \cdot \det M(\ell,g) \cdot \left(D^1_w z\right)^{\binom{g+1}{2}} (D^2_w z)^{\ell(\ell-g)}
    \]
    where $D^i_wz= \frac{D^i z}{dw}$.
\end{thm}

\begin{thm}\label{thm:euler_index_ramif_l>g_infty}
Let $X$ denote a hyperelliptic curve defined over a field $F$ of characteristic $p \neq 2$ as above. Whenever $\ell>g$, $2\ell - g \equiv 1 \pmod{4}$, and $\det M(\ell,g)$ is nonzero in $F$, the local Euler index of the complete linear series $|2\ell \infty_X|$ in ${\rm GW}(F)$ associated at $\infty_X$ is given by
\begin{align}
    &\mathrm{ind}_{\infty_X}w(\lambda) \notag \\
    &=\begin{cases}
            \dfrac{1}{2}\binom{g+1}{2}\cdot \mb{H} & \text{if $\ell$ is even and $g \equiv 3 \pmod{4}$}\\
            \noalign{\vskip9pt}
            \dfrac{\binom{g+1}{2}-1}{2}\cdot \mb{H} + \left< (-1)^{1+\binom{\ell-g}{2}} 2(\det M(\ell,g)) (D^1h)(0) \right> & \text{if $\ell$ is odd and $g \equiv 1 \pmod{4}$}.
        \end{cases} \label{eq:Euler_ramif_l>g_infty}
\end{align}
\end{thm}

\subsubsection{\bf Example: $\ell=1$} In this case, all local inflectionary indices away from $R_{\pi}$ are zero. The local index in a hyperelliptic ramification point $(\ga,0)$ other than $\infty_X$ is $\mbox{Tr}_{k(\ga)/F} \langle \frac{2}{f^{\pr}(\ga)} \rangle$, which yields $\mbox{Tr}_{k(\ga)/F} \langle 1 \rangle$ whenever $\frac{2}{f^{\pr}(\ga)}$ is a square in $k(\ga)$. The local index in $\infty_X$, on the other hand, is $\left< -\frac{2}{h^{\pr}(0)} \right>$, where $h(z):=z(z^{2g+1}f(z^{-1}))$. It is worth noting that in \cite[\S 11]{L1} and \cite{BKW}, $\mathrm{ind}_p(s)$ and $e(\mc{E})$ were computed by comparing the canonical bundle of $X$ against that of $\mathbb{P}^1$. 


\subsection{Local Euler indices away from $R_{\pi}$, and Hasse inflection polynomials}\label{subsec:hasse_inflection_polys}

In this subsection, we will generalize the description of local inflection indices given in \cite{CG} and \cite{CG1} to arbitrary characteristic. Accordingly, given positive integers $\ell>g$, we define the {\it $(g,\ell)$th Hasse inflection polynomial} $P_{g,\ell}(x) \in F[x]$ according to
\begin{equation}\label{hasse_inflection_poly}
\det (D^{j}x^iy)_{0 \leq i \leq \ell-g-1; \ell+1 \leq j \leq 2\ell-g}=(f^{-(\ell+1)}y)^{\ell-g}P_{g,\ell}(x)
\end{equation}
where $D^{j}=D^{j}_x$. The characteristic property of $P_{g,\ell}$ is that its roots parameterize precisely the $x$-coordinates of $\overline{F}$-rational Hasse inflection points of the complete linear series $|2\ell \infty_X|$ on $X$ {\it supported on the complement of $R_{\pi}$}. It is worth noting that when $\ell=g+1$, the equation \eqref{hasse_inflection_poly} reduces to the statement that
\begin{equation}\label{eq:infl_l=g+1}
D^{g+2}y= f^{-(g+2)}y \cdot P_{g,g+1}(x).
\end{equation}
In general, we can always realize Hasse inflection polynomials as determinants in the ``atomic" polynomials $P_{g,g+1}(x)$, according to the following Hasse analogue of \cite[Lem. 2.1]{CG1}.

\begin{prop}\label{inflection_poly_determinantal_presentation}
Given positive integers $\ell>g$, let $\mu:=\ell-g$ and define a homogeneous polynomial of degree $\mu$ in the variables $t_{-\mu},\dotsc,t_0,\dotsc,t_{\mu-1}$ as follows:
\[
Q_{\mu,\ell+1}=\det(t_{j-i})_{0 \leq i,j \leq \mu-1} \in \mb{Z}[t_{-\mu},\dots,t_0,\dots,t_{\mu-1}].
\]
We then have $P_{g,\ell}=Q_{\mu,\ell+1}|_{t_i=P_{\ell+1+i,\ell+2+i}}$.
\end{prop}

\begin{proof}
The desired result follows immediately from the (first part of the) Hasse analogue of \cite[Rmk 3.5]{BCG}; namely, that away from the ramification locus $R_{\pi}$, the Hasse Wronskian is locally given by
\[
w(\la)= \det(D^{\ell+1+j-i}y)_{1 \leq i,j \leq \ell-g}.
\]
The latter equality results from an easy, if slightly tedious, row reduction that we leave to the reader.
\end{proof}

The atomic inflection polynomials $P_{g,g+1}(x)$, in turn, may be calculated recursively. Namely, given a positive integer $g$, let $n=g+2$; for simplicity (and consistently with \cite{CG}) we write $P_n$ in place of $P_{g,g+1}$. The following result is a Hasse analogue of \cite[Sect. 2, eq. (3)]{CG}.
\begin{prop}\label{prop:recursion}
Suppose that $\mathrm{char}(F) \neq 2$. The atomic inflection polynomials of the hyperelliptic curve defined by the affine equation $y^2=f(x)$ may then be obtained by applying the recursive relation
\begin{equation}\label{eq:reduced_recursion}
P_{n+1}= \frac{1}{n+1}\bigg((D^1P_n)\cdot f+ \bigg(-n+\frac{1}{2}\bigg) P_n \cdot (D^1f)\bigg)
\end{equation}
for every $n \geq 1$, subject to the seed datum $P_1=\frac{1}{2}D^1f$.
\end{prop}

\begin{proof}
    Assume first that $\mathrm{char}(F)=0$. Note that differentiating  $y^2=f(x)$ yields $2yD^1y=D^1f$. 
    Similarly, differentiating both sides of the defining equation~\eqref{eq:infl_l=g+1} for atomic inflection polynomials yields
    \begingroup
    \allowdisplaybreaks
    \begin{align*}
        D^1D^{n}y&=(D^1P_n)f^{-n}y+P_n\cdot (-nf^{-(n+1)}(D^1f)\cdot y+f^{-n}D^1y)\\
        &=(D^1P_n)f^{-n}y+P_n\cdot \left(-nf^{-(n+1)}(D^1f)\cdot y+f^{-n}\cdot \frac{1}{2}f^{-1}(D^1f)\cdot y\right)\\
        &=f^{-(n+1)}y\left( (D^1P_n)\cdot f+P_n\cdot (D^1f)\cdot \left(-n+\frac{1}{2} \right) \right).
    \end{align*}
    \endgroup
    The desired recursion now follows from the fact that $D^1D^{n}=(n+1)D^{n+1}$; see \S\ref{subsec:Hasse_Witt}.


\medskip
Assume now that the base field $F$ has characteristic $p>2$. The recursion \eqref{eq:reduced_recursion} makes sense provided $p \not | (n+1)$, as does the argument used in proving it. More generally, as we now explain, $P_{n+1}$ may be obtained by lifting to characteristic zero and then specializing to the base field $F$. The upshot is that there is a well-defined way to ``divide'' by $n+1$.

\medskip
Specifically, we claim that $P_{n+1}$ over $F$ is the specialization of an inflection polynomial defined over a discrete valuation ring $R$ whose field of fractions has characteristic zero, associated to a ``spreading out" of the underlying $F$-curve $y^2=f(x)$ over $R$. Clearly such an $R$ exists; moreover, the argument we used to prove the recursion \eqref{eq:reduced_recursion} 
in the characteristic zero case extends to $R$ and yields
\begin{equation}\label{eq:genl_recursion}
    (n+1)P_{n+1}= (D^1P_n)\cdot f+ \bigg(-n+\frac{1}{2}\bigg) P_n \cdot (D^1f).
\end{equation}
Since Hasse differentiation commutes with base change (see \cite[Prop. 5.6]{Vojta}), equation~\eqref{eq:genl_recursion} descends to $F$. 
In particular, the $p$-adic multiplicity of the right side of equation~\eqref{eq:genl_recursion} is at least that of $n+1$; this means in turn that the right side of 
equation~\eqref{eq:reduced_recursion} is well-defined over $F$ whenever $p>2$, provided we divide both sides of \eqref{eq:genl_recursion} by suitable powers of $p$ (here $p$ is a power of the uniformizer of $R$ up to multiplication by a unit). 
Finally, it is straightforward to check that the recursion we obtain is independent of our choice of mixed-characteristic extension.
\end{proof}

By construction, the Hasse Wronskian of the complete linear series $|\mc{O}_X(2\ell\infty_X)|$ on the hyperelliptic curve $X$ of genus $g$ is $w(\lambda)=(-1)^{\binom{\ell-g+1}{2}} (f^{-(\ell+1)}y)^{\ell-g} P_{g,\ell}(x)$ with respect to the local \'etale coordinate $x$, where $\lambda = (1,y,\dotsc,x^{\ell-g-1},x^{\ell-g-1}y;x^{\ell-g},x^{\ell-g+1},\dotsc,x^{\ell})$, and $(-1)^{\binom{\ell-g+1}{2}}$ is the sign of the permutation that reorders $\lambda$ as $(1, x, x^2, \dotsc, x^{\ell}; y, xy, \dotsc, x^{\ell-g-1}y)$. In particular, every root $x=\ga$ of an inflection polynomial $P_{g,\ell}(x)$ lifts to two inflection points $(\ga,\pm \sqrt{f(\ga)})$ on $X$. As detailed in \S \ref{subsec:Nis}, computing the local Euler index of $w(\lambda)$ at $(\ga,\pm \sqrt{f(\ga)})$ involves identifying the leading term of the power series expansion of $w(\lambda)$ with respect to a local Nisnevich coordinate. In full generality, this is somewhat delicate; there are two cases, depending on whether the extension $k(\ga,\pm \sqrt{f(\ga)})/k(\ga)$ is trivial or not. When the extension is trivial (resp., nontrivial), 
the point $\ga \in \mb{A}^1_x \subset \mb{P}^1$ {\it splits} (resp., is {\it inert}) with respect to the hyperelliptic projection $\pi: X \rightarrow \mb{P}^1$ at $\ga \in \mb{A}^1_x \subset \mb{P}^1$. Note that $x$ determines a local Nisnevich coordinate at $(\ga, \pm \sqrt{f(\ga)})$ if and only if $\ga$ splits. This fact 
is crucial in determining the local Euler index at $(\ga, \pm \sqrt{f(\ga)})$; see Remark~\ref{rmk:infl_poly_cond} below for a discussion of the inert case.
\begin{thm}\label{thm:local_Euler_index_infl_poly}
    Assume that $\ell > g$, $2\ell - g \equiv 1 \pmod{4}$, and $\mathrm{char}(F) \neq 2$. Choose a square root $\sqrt{f(\ga)}$ of $f(\ga)$, where $\ga$ is an $x$-coordinate of an inflection point of a complete linear series $|\mc{O}_X(2\ell\infty_X)|$ on a hyperelliptic curve $X$ of genus $g$. If $\sqrt{f(\ga)} \in k(\ga)$ (i.e., $\ga$ splits), then 
    \[
    \mathrm{ind}_{(\gamma,\sqrt{f(\ga)})}w(\lambda)
    =\mathrm{Tr}_{k(\ga)/F}\left((-1)^{\binom{\ell-g+1}{2}}(f^{-(\ell+1)}(\ga)\sqrt{f(\ga)})^{\ell-g}\cdot \mathrm{ind}_{(\gamma,0)}^{k(\ga)}P_{g,\ell}(x)\right)
    \] 
    and $\mathrm{ind}_{(\gamma,-\sqrt{f(\ga)})}w(\lambda)=\mathrm{Tr}_{k(\ga)/F}\left((-1)^{\ell-g} \cdot \mathrm{ind}_{(\gamma, \sqrt{f(\gamma)})}^{k(\ga)}w(\lambda)\right)$, where $\mathrm{ind}^{k(\ga)}$ is the local Euler index over the field $k(\ga)$.
\end{thm}

\begin{proof}
    The condition $\sqrt{f(\ga)} \in k(\ga)$ implies that the residue field of $(\ga, \pm \sqrt{f(\ga)})$ is $k(\ga)$, hence $x$ is itself a local Nisnevich coordinate near those inflection points. The proposition now follows from our discussion (and proof of) Proposition \ref{inflection_poly_determinantal_presentation}.
\end{proof}

When $\ell=g+1$, the inflection polynomial is of the form $P_n(x)$ for some $n \geq 3$, and may be computed recursively as in Proposition~\ref{prop:recursion}; the general inflection polynomial $P_{g,\ell}(x)$ may then be computed as a determinant in the atomic inflection polynomials $P_n(x)$ as in Proposition~\ref{inflection_poly_determinantal_presentation}. The defining equation \eqref{hasse_inflection_poly} for inflection polynomials reduces the calculation of the local Euler index in one of the two $\overline{F}$-preimages $(\ga,\pm \sqrt{f(\ga)})$ of $x=\ga$ to a calculation purely in terms of $x$. So whenever we understand the power series expansion of all $P_n(x)$ at $x=\ga$, we may deduce the local Euler index of $w(\lambda)$ at inflection points $(\ga, \pm \sqrt{f(\ga)})$.

\begin{rem}\label{rmk:infl_poly_cond}
Note that the condition $\sqrt{f(\ga)} \in k(\ga)$ (i.e., $\ga$ splits) is necessary in Theorem~\ref{thm:local_Euler_index_infl_poly} in order to obtain a local Euler index formula that depends only on $\gamma, f$, and $\ell$. To see the difficulty in computing the local Euler index when $\sqrt{f(\ga)} \not \in k(\ga)$ (i.e., $\ga$ is inert), first assume that $k(\gamma, \sqrt{f(\ga)})/F$ is a separable extension as in Remark~\ref{rmk:perfect_field}. In this case, $x$ is no longer a local Nisnevich coordinate; 
rather, we need to replace it by a general $F$-linear combination $ax+by$ corresponding to a general linear projection. The transition map that sends the local \'etale coordinate $x$ to the local Nisnevich coordinate $ax+by$ is $(D^1_{ax+by}x)^{\binom{2\ell-g+1}{2}}$, which is the reciprocal of
\[
    (D^1_x(ax+by))^{\binom{2\ell-g+1}{2}}=(a+bD^1_yx)^{\binom{2\ell-g+1}{2}}.
\]
Note that near $(\gamma, \sqrt{f(\ga)})$, we have $D^1_yx=(D^1_xy)^{-1}=(D^1_x\sqrt{f(x)})^{-1}=\frac{2\sqrt{f(x)}}{f^{\pr}(x)}$. Nonetheless, the evaluation of $(D^1_{ax+by}x)^{\binom{2\ell-g+1}{2}}$ at $(\ga,\sqrt{f(\ga)})$ depends on $a$ and $b$, which makes it difficult to extract a local Euler index formula independent of our choice of $(a,b)$ (even if the local Euler index does not itself depend on the choice of $(a,b)$). This suggests that it only makes sense to work with such a formula provided we assume $F$ to be a 
specifically-chosen field, as in \S \ref{sec:geometric_interpretations}.
\end{rem}

\section{Geometric interpretations of Euler indices}\label{sec:geometric_interpretations}
The objective of this section is twofold. We first give concrete interpretations for several interesting base fields (not of characteristic 2) of the global Euler index computed by Theorem~\ref{thm:global_Euler_class_inflection}, as well as of the local Euler indices for points that belong to the ramification locus $R_{\pi}$ of the hyperelliptic projection calculated in Theorems~\ref{thm:euler_index_ramif_l_leq_g} and \ref{thm:euler_index_ramif_l>g}. We then give analogous concrete interpretations and speculations for points in the complement of $R_{\pi}$ when the underlying curve is of genus one.

\subsection{Euler indices over $\mb{R}$} The study of inflection points of real {\it plane} curves dates back to Klein \cite{Kle}, who showed 
that a degree-$d$ curve in $\mb{P}^2$ defined over $\mb{R}$ has at most $d(d-2)$ real inflection points.\footnote{In particular, this means that at most one third of the inflection points are defined over $\mb{R}$; this bound, moreover, is known to be sharp.} For non-planar curves and linear series of rank greater than 2, few quantitative estimates for $\mb{R}$-rational inflection points exist, though \cite{KS} and \cite{BCG} address the rational and hyperelliptic cases, respectively. In the latter two papers, the authors sought to maximize $\mb{R}$-inflection by controlling the (global) topology of the real locus; here, our aim is rather to point out consequences of our arithmetic inflection formula that are uniform across {\it all} topological types. 

\medskip
As explained in the introduction, we have $\rm{GW}(\mb{R}) \cong \mb{Z}^2$, in which the two $\mb{Z}$ factors correspond, respectively, to the rank and signature of quadratic forms. Computing the rank of the classes in Theorem~\ref{thm:global_Euler_class_inflection} (resp., in Theorems~\ref{thm:euler_index_ramif_l_leq_g} and \ref{thm:euler_index_ramif_l>g}), we clearly recover the total $\mb{C}$-inflection of $|2\ell \infty_X|$ (resp., the $\mb{C}$-inflection multiplicity of $|2\ell \infty_X|$ over $R_{\pi}$). On the other hand, the fact that ${\rm sign}(\mb{H})=0$ implies that the signature of the global Euler class is zero. In light of Remark~\ref{rem:loc_Euler_cases}, this means precisely that the sum of oriented Milnor indices over all inflection points of $|2\ell \infty_X|$ is zero. 


\medskip
The signatures of the local Euler indices at ramification points away from $\infty_X$ in Theorems~\ref{thm:euler_index_ramif_l_leq_g} and \ref{thm:euler_index_ramif_l>g} depend on the field of definition of the $x$-coordinates $\ga$ of the underlying ramification points $(\ga,0)$. If $\ga \in \mb{C} \setminus \mb{R}$, then exactly as in \cite[Rmk 6.2]{McK}, the signature of the local Euler index is zero. On the other hand, if $\ga \in \mb{R}$ and $\ell \leq g$, the signature of the local Euler index is equal to the sign of $f^{\pr}(\ga)$. An analogous and only slightly more elaborate statement holds when $\ga \in \mb{R}$ and $\ell > g$, in which case the local Euler index is zero when $\binom{g+1}{2}$ is even, and equal to the sign of $(f^{\pr}(\ga))^{\ell(\ell-g)+1}$ when $\binom{g+1}{2}$ is odd. 

\medskip
On the other hand, according to Theorems~\ref{thm:euler_index_ramif_l_leq_g_infty} and \ref{thm:euler_index_ramif_l>g_infty}, the signature of the local Euler index at $\infty_X$ depends on the sign of $h^{\pr}(0)$, where $h(z)=z^{2g+2}f(z^{-1})$; moreover, $h^{\pr}(0)$ is exactly the leading coefficient of $f$.  If $\ell \leq g$, then the signature of the local Euler index is equal to the sign of $-h^{\pr}(0)$. Analogous arguments show that when $\ell > g$, the local Euler index is zero when $\binom{g+1}{2}$ is even, and otherwise it is equal to the sign of $(-1)^{1+\binom{\ell-g}{2}}(h^{\pr}(0))^{1+\ell(\ell-g)}$.

\medskip
This means, in particular, that whenever $f$ is {\it monic} and $\ell \le g$, the difference between the number of ramification points $\gamma \in \mb{R}$ for which $f^{\pr}(\ga)$ is negative and the number of ramification points for which $f^{\pr}(\gamma)$ is positive is precisely one. 


\subsection{Euler indices over $\mb{F}_q$}\label{subsec:Euler_finite} 
We have ${\rm GW}(\mb{F}_q) \cong \mb{Z} \times \mb{Z}/2\mb{Z}$, 
in which the factors correspond, respectively, to the rank and the discriminant. More precisely, elements of ${\rm GW}(\mb{F}_q)$ of given rank are classified according to whether or not their discriminants are squares; 
note that for explicit calculations we think of the second factor as $\mb{F}_q^{\ast}/(\mb{F}_q^{\ast})^2$ with its multiplicative group law.
The discriminant of the global Euler class computed by Theorem~\ref{thm:global_Euler_class_inflection} is $(-1)^{\frac{\ga_{\mb{C}}}{2}}=(-1)^{\frac{g(2\ell-g+1)^2}{2}}$, which is a square exactly when $g(2\ell-g+1)^2$ or $(q-1)$ is divisible by 4, i.e., if and only if $g$ is odd or $g$ is divisible by 4 or $(q-1)$ is divisible by 4. To compute the discriminants of the 
local Euler indices at ramification points $\gamma \in \mb{A}^1_{\mb{F}_q}$ in Theorems~\ref{thm:euler_index_ramif_l_leq_g} and \ref{thm:euler_index_ramif_l>g}, we proceed as in \cite[\S 6.3]{McK}. Namely, we apply the facts that ${\rm Tr}_{k(\ga)/\mb{F}_q} \langle \al \rangle= \langle {\rm norm}(\al) \rangle \cdot {\rm Tr}_{k(\ga)/\mb{F}_q} \langle 1 \rangle$ for every $\al \in k(\ga)$; that ${\rm norm}(\al)$ is a square in $\mb{F}_q$ if and only if $\al$ is a square in $k(\ga)$; and that ${\rm Tr}_{k(\ga)/\mb{F}_q} \langle 1 \rangle$ is a square if and only if $k(\ga)=\mb{F}_{q^n}$ with $n$ odd. More generally, given $\beta = \sum \langle \alpha_i \rangle \in \rm{GW}(k(\gamma))$, the corresponding trace ${\rm Tr}_{k(\ga)/\mb{F}_q} \beta = \sum \langle {\rm norm}(\alpha_i) \rangle \cdot {\rm Tr}_{k(\ga)/\mb{F}_q} \langle 1 \rangle$ has discriminant ${\rm disc} ({\rm norm}(\prod \alpha_i)) \cdot {\rm disc}({\rm Tr}_{k(\ga)/\mb{F}_q}\langle 1 \rangle )^{\mathrm{rank}(\beta)}$.

\medskip
When $\ell \leq g$, the local index $\beta = \sum \langle \alpha_i \rangle$ computed by Theorem~\ref{thm:euler_index_ramif_l_leq_g} satisfies $\prod \alpha_i = (-1)^{\frac{\binom{\ell+1}{2}-1}{2}}\frac{f^{\pr}(\gamma)}{2}$. The arguments of the preceding paragraph now show that the local index depends only on whether $\frac{f^{\pr}(\gamma)}{2}$ is a square in $k(\gamma)$, together with certain parity conditions on $\ell$ and the degree of $k(\gamma)/\mb{F}_q$. When $\ell > g$, we further assume that $\det M(\ell,g)$ is nonzero in $\mb{F}_q$ so that Theorem~\ref{thm:euler_index_ramif_l>g} applies; the local index at $\gamma$ then depends on whether $2$ and $f^{\pr}(\gamma)$ is a square in $k(\gamma)$, along with parity conditions on $\ell$ and $g$ and the degree of $k(\gamma)/\mb{F}_q$.

\medskip
Similarly, when $\ell \le g$, the local index at $\infty_X$ depends only on whether $\frac{h^{\pr}(0)}{2}$ is a square in $\mb{F}_q$, together with certain parity conditions on $\ell$. When $\ell > g$ and $\det M(\ell,g)$ is nonzero in $\mb{F}_q$, the local index at $\infty_X$ depends on whether or not $2$ and $h^{\pr}(0)$ are squares in $\mb{F}_q$, along with parity conditions on $\ell$ and $g$.

\medskip
Summing local Euler indices and assuming that $p\equiv 1 \pmod{8}$ (so that $-1$ and $2$ are squares in $\mb{F}_q$), we see, for example, that when $\ell\leq g$ and $f$ is monic, the difference between the number of ramification points $\gamma \in \overline{\mb{F}}_q$ for which the discriminant is not a square (either $f^{\pr}(\gamma)$ is a not a square in $k(\gamma)$ and the degree of $k(\gamma)/\mb{F}_q$ is odd, or $f^{\pr}(\gamma)$ is a square and the degree of $k(\gamma)/\mb{F}_q$ is even) and the number of ramification points $\gamma \in \overline{\mb{F}}_q$ for which the discriminant is a square is 1 modulo 2.


\subsection{Euler indices over $\mb{C}(\!(t)\!)$} We have ${\rm GW}(\mb{C}(\!(t)\!)) \cong \mb{Z} \times \mb{Z}/2\mb{Z}$, in which the factors correspond, respectively, to the rank and the discriminant; see, e.g., \cite[Prop. 6.9]{McK}. Let $\nu=\nu_t$ denote the standard $t$-adic valuation on $\mb{C}(\!(t)\!)$ given by the least non-vanishing degree in $t$. If $g = ut^{\nu(g)}$ for some unit $u$, then $\langle g\rangle$ is $\langle 1\rangle$ (resp., $\langle t\rangle$) when $\nu(g)$ is even (resp., odd). 
Since $\mathbb{H} = 2\langle 1\rangle = 2\langle t\rangle$, it follows that the discriminant of the global Euler class computed by Theorem~\ref{thm:global_Euler_class_inflection} is $0$. 

\medskip
An explicit description of the discriminants of the local Euler indices in Theorems~\ref{thm:euler_index_ramif_l_leq_g} and \ref{thm:euler_index_ramif_l>g} is as follows. When $\ell\leq g$ and $(\gamma,0)$ is a ramification point defined over $\mb{C}(\!(t^{1/m})\!)$, the corresponding local index depends on the parity of the $t^{1/m}$-adic valuation of $f^{\pr}(\gamma)$. More explicitly, $\mathrm{disc}(\mathrm{ind}_{(\gamma,0)})$ is given by ${\rm Tr}_{\mb{C}(\!(t^{1/m})\!) / \mb{C}(\!(t)\!)}\langle 1\rangle$ (resp., ${\rm Tr}_{\mb{C}(\!(t^{1/m})\!) / \mb{C}(\!(t)\!)}\langle t^{1/m}\rangle$) when the valuation is even (resp., odd). Similarly, when $\ell > g$, $\binom{g+1}{2}$ is odd, and $(\gamma,0)$ is a ramification point defined over $\mb{C}(\!(t^{1/m})\!)$, we see that $\mathrm{disc}(\mathrm{ind}_{(\gamma,0)})$ is given by ${\rm Tr}_{\mb{C}(\!(t^{1/m})\!) / \mb{C}(\!(t)\!)}\langle 1\rangle$ (resp., ${\rm Tr}_{\mb{C}(\!(t^{1/m})\!) / \mb{C}(\!(t)\!)}\langle t^{1/m}\rangle$) if the valuation of $f^{\pr}(\gamma)^{1+\ell(\ell-g)}$ is even (resp., odd); and that $\mathrm{disc}(\mathrm{ind}_{(\gamma,0)})$ always vanishes when $\binom{g+1}{2}$ is even. The above analysis goes through verbatim if we replace $(\gamma,0)$ by $\infty$, $m$ by 1, and $f^{\pr}(\gamma)$ by $h^{\pr}(0)$.

\medskip
The value of these traces, in turn, only depends on the parity of $m$. Indeed, proceeding as in the $\mb{F}_q$ case, we see that ${\rm Tr}_{\mb{C}(\!(t^{1/m})\!) / \mb{C}(\!(t)\!)}\langle 1\rangle \equiv m-1 \pmod{2}$ and ${\rm Tr}_{\mb{C}(\!(t^{1/m})\!) / \mb{C}(\!(t)\!)}\langle t^{1/m}\rangle \equiv m \pmod{2}$.

\medskip
Summing local Euler indices, we deduce that when $\ell\leq g$ and $f$ is monic, the difference between the number of ramification points $(\gamma,0)$ with odd discriminant (either the degree $m$ of the field of definition ${\mb{C}}(\!(t^{1/m})\!)$ is odd and $f^{\pr}(\gamma)$ has odd valuation, or $m$ is even and $f^{\pr}(\gamma)$ has even valuation) and the number of ramification points $\gamma \in \overline{\mb{C}(\!(t)\!)}$ with even discriminant is 1 modulo 2.


\subsection{Inflection polynomials for elliptic curves}\label{inflectionary_curves}
Over the complex numbers, the {\it Legendre} family of elliptic curves parameterized by
\begin{equation}\label{legendre_family}
y^2=x(x-1)(x-\kappa)
\end{equation}
leads to a nice presentation of the moduli stack of (marked) elliptic curves. In \cite{CG,CG1} the authors studied the variation of inflection polynomials in the Legendre parameter $\tau$.\footnote{In this situation, the linear series on the underlying curves are no longer generally complete; rather, they are distinguished codimension-$(g-1)$ subseries of complete series of degree $2\ell$ generated by the monomial bases $\la$ in the toric coordinates $x$ and $y$ introduced previously.} Their roots, as $\kappa$ varies, trace out {\it inflectionary} curves in the $(x,\kappa)$-plane. According to the analogy between torsion and inflection introduced in \S \ref{subsec:generalized_torsion}, we may think of these inflectionary curves as generalizations of modular curves. Whenever $\kappa$ is a {\it real} parameter, so that the corresponding elliptic curve $E_{\kappa}$ has a real locus $E_{\kappa}(\mb{R})$ with two connected components, \cite[Conj. 3.1]{CG1} predicts that the corresponding inflection polynomial is separable, i.e., has only simple roots.

\begin{figure}
\includegraphics[scale=0.5]{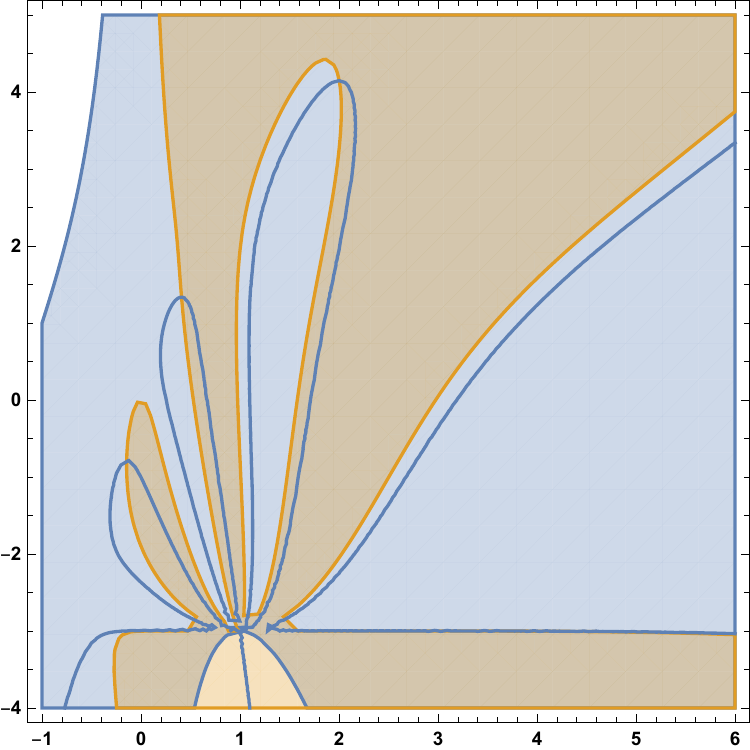}
\includegraphics[scale=0.5]{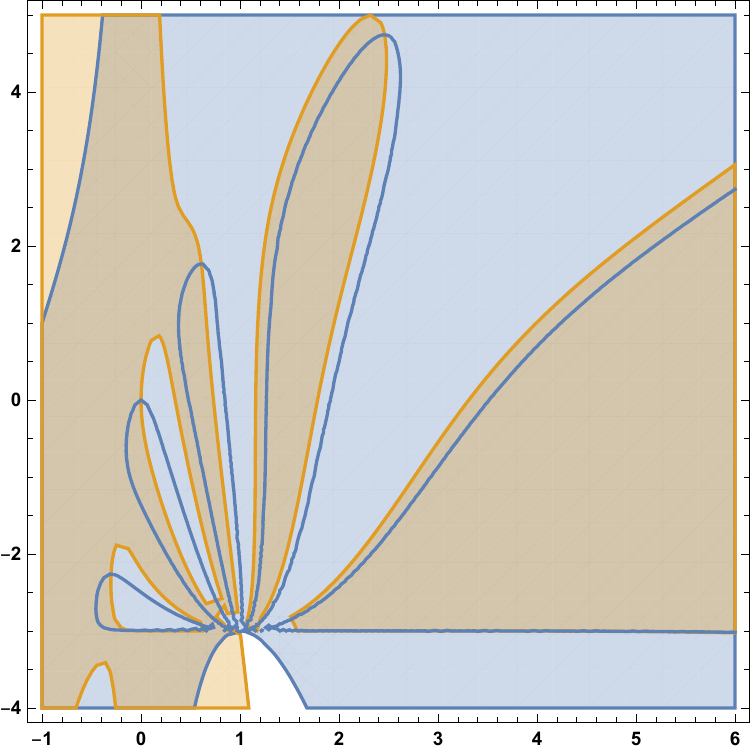}
\caption{Dark blue curves trace out the real loci of $(P_n=0)$ for $n=9,10$ in the $(x,a)$-plane. Here $a$ parameterizes the punctured $j$-line, and the fiber over $a$ is the elliptic curve $E_{(a,2)}:z^2=x^3+ax+2$ in the $(x,z)$-plane. Grey (resp., orange) shading indicates that the Weierstrass cubic $f(x)=x^3+ax+2$ (resp., $\frac{dP_n}{dx}$) is strictly positive.}
\label{Weierstrass_inflectionary_curves}
\end{figure}

\medskip
In \cite{H}, meanwhile, J. Huisman studies the moduli of elliptic curves over base fields $F$ of characteristic $\text{char}(F) \neq 2,3$ via their {\it Weierstrass} presentations $E_{(a,b)}: y^2= x^3+ax+b$. He shows that isomorphism classes of $\mb{R}$-elliptic curves are parameterized by a real projective $j$-line that is an $\mb{R}^{\ast}$-quotient of the punctured plane $\mb{R}^2 \setminus \{(0,0)\}$ with coordinates $(a,b)$. The $j$-line, in turn, is stratified according to the sign of the discriminant $\Delta(a,b)=-16(4a^3+27b^2)$; elliptic curves $E_{(a,b)}$ of strictly positive (resp., negative) discriminant are those whose real loci split as two connected components (resp., comprise a single connected component). Note that when $\Delta=0$, the corresponding fiber $E_{(a,b)}$ is a nodal rational curve. We now turn to the behavior of inflection polynomials along the open sublocus of the $j$-line where $b$ is nonzero. As a convenient normalization, we set $b=2$; then $a$ becomes a local parameter for our punctured $j$-line, and the ray $(a<-3)$ (resp., $(a>-3)$ parameterizes those isomorphism classes of elliptic curves with strictly positive (resp., negative) discriminants. The behavior of the negative-discriminant regime $(a>-3)$ displays some distinctive features relative to that of the positive-discriminant case explored in \cite{CG1}. 

\medskip
First of all, the number $\omega_{\mb{R}}$ of real inflection points is no longer uniform in the modular parameter $a$; rather, there are several critical intervals in $a$ along which $\omega_{\mb{R}}$ is constant, corresponding to the petals in Figure~\ref{Weierstrass_inflectionary_curves}. Computer experiments indicate that the number of petals, as a function of $n$, is precisely $\frac{n}{2}-1$ (resp., $\frac{n-1}{2}-1$) when $n$ is even (resp., odd). The {\it peak} of each petal, on the other hand, belongs to the discrete set of points in the $(x,a)$-plane for which $P_n= \frac{dP_n}{dx}=0$, where $P_n$ is the inflectionary polynomial associated to the Weierstrass cubic $y^2=x^3+ax+2$ (see \S~\ref{subsec:hasse_inflection_polys}); indeed, a peak is tautologically a point where $\frac{da}{dx}=0$, and $\frac{dP_n}{dx}=0$ follows by the chain rule. 
Similarly, those $a$-values at which (the corresponding specialization of) $P_n$ fails to be separable are roots of the $x$-discriminant of $P_n$. In particular, the non-separable set of $a$-values includes the $a$-coordinates of the petal peaks mentioned above. Our empirical (computer-based) evidence suggests the following is true.

\begin{conj}\label{atomic_inflectionary_curves}
Let $a \in \mb{R}$, and let $P_n(x)$, $n \geq 2$ denote the $n$th atomic inflection polynomial associated to the real Weierstrass elliptic curve $E_{(a,2)}: y^2=x^3+ax+2$ as above. The possible numbers of real zeroes of $P_n(x)$, as a set-valued function of the modular parameter $a$, are as follows.
\begin{center}
    \begin{tabular}{|p{3cm}|p{3cm}|p{3cm}|}
    \hline
        Value of $a$ & $n$ odd & $n$ even \\
     \hline
     $a<-3$ & 4, of which 2 satisfy $f>0$ & 2, of which 1 satisfies $f>0$
     \\
     \hline
     $a>-3$ & $2i, i=1, \dots, \frac{n-1}{2}$, of which $(2i-1)$ satisfy $f>0$ & $2i, i=1, \dots, \frac{n}{2}$, of which $(2i-1)$ satisfy $f>0$ \\
     \hline
    \end{tabular}
\end{center}
The number of real roots (in $x$) of $P_n$ is nonincreasing as a function of $a$, as $a$ increases from $-3$ to $\infty$. Moreover, when $n$ is even (resp., odd) there are precisely $\frac{n}{2}-1$ (resp., $\frac{n-1}{2}-1$) values of $a$ at which $P_n$ fails to be separable, and these are the $a$-coordinates of the petal peaks described above. At each of these $a$-values, $P_n$ has a double root and all other roots are simple.
\end{conj}

Note, in particular, that the conjecture predicts that for every fixed value of $a$, $P_n$ has an {\it even} number of real roots; and that whenever $n$ is {\it even}, there are values of $a$ for which $P_n$ has {\it only} real roots. Another counter-intuitive upshot of Conjecture~\ref{atomic_inflectionary_curves} is that for fixed values of $n$, those real elliptic curves whose linear series are maximally $\mb{R}$-inflected are not necessarily those with the maximal number (two) of real components. On the other hand, the precise distribution of {\it signatures} of the roots of $P_n$ appears to be somewhat intricate, as evidenced in Figure~\ref{Weierstrass_inflectionary_curves} by the pattern in which the petals intersect the domains of positivity for $\frac{dP_n}{dx}$.

\medskip
It is natural to ask for analogues of Conjecture~\ref{atomic_inflectionary_curves} over $\mb{Q}$ or $\mb{F}_q$.
Faltings' theorem implies that the number of $\mb{Q}$-rational points of the inflectionary curve $\mc{C}_n:=(P_n=0)$ is finite; and in principle the Chabauty--Coleman method of $p$-adic integration \cite{KRZB, McP} may be used to compute $\mc{C}_n(\mb{Q})$. Implementation, however, is beyond the scope of the current paper.

\medskip
Over $\mb{F}_q$, on the other hand, Hasse--Weil theory \cite{Se} applies, and establishes that $$\#\mc{C}_n(\mb{F}_q)=q+1+e_{n,q},$$ where $e_{n,q}$ is a bounded error term.
More precisely, Deligne showed in \cite[Thm 8.1]{D} that $|e_{n,q}| \leq 2g\sqrt{q}$ for {\it smooth} curves of genus $g$. In our case, $\mc{C}_n$ is singular for every $n \geq 2$, but Aubry and Perret \cite{AP} showed that the same basic inequality holds, provided we interpret $g$ as the arithmetic genus.
Note that $\mc{C}_n$ is always a plane curve of degree $2n$. Accordingly, it is instructive to examine the distribution of the renormalized error terms $\wt{e}_{n,p}:=\frac{e_{n,p}}{(2n-1)(2n-2)\sqrt{p}}$ associated with the inflectionary curves $\mc{C}_n$ derived from either the Legendre or Weierstrass presentations as $p$ varies over all odd primes $p$, in the spirit of the Sato--Tate conjecture for elliptic curves proved in \cite{BLGHT}.\footnote{If we further assume that $p \equiv 2 \text{ (mod }3)$, then every element of $\mb{F}_p$ is a cube, and it then follows from Cardano's cubic formula that an elliptic curve $E_{(a,2)}$ admits a Legendre presentation if and only if its discriminant $\Delta=\Delta(a,2)$ is a square in $\mb{F}_p$.}

\medskip
Somewhat counter-intuitively, (the $\mb{F}_p$-rational behavior of) the inflectionary curves $\mc{C}_n$ depends strongly upon whether the underlying family of elliptic curves is of Legendre or Weierstrass type. Most strikingly, the delta-invariants of the singularities and the geometric genera of Legendre and Weierstrass inflectionary curves $\mc{C}_n$ are generally distinct. The singularities and geometric genera of Legendre inflectionary curves $\mc{C}_n$ were addressed in detail in \cite{CG1}; the most salient points (which are conjectural in general, but hold when $n$ is small) are that a) $\mc{C}_n$ has precisely three singularities, each of which is defined over $\mb{Q}$, and which are permuted by automorphisms of $\mc{C}_n$; and b) each singularity of $\mc{C}_n$ has delta-invariant $\lfloor \frac{(n-1)^2}{2} \rfloor+ n-1$. In particular, this implies that the Legendre inflectionary curve $\mc{C}_n$ is of geometric genus {\it zero} whenever $2 \leq n \leq 5$, provided $\mc{C}_n$ is irreducible.\footnote{This might seem to suggest that $\mc{C}_3$ is of geometric genus $10- 3(4)=-2$, an absurd conclusion; the explanation, however, is that $\mc{C}_3$ in fact has 3 irreducible components, so the geometric genus is two more than we expected.} Moreover, we expect that $\mc{C}_n$ is almost always irreducible; the unique exception to this rule of which we are aware is provided by $P_3$, which factors over $\mb{Q}$ as
\[
P_3= \frac{1}{16} (\kappa-x^2)(\kappa-2x+x^2)(\kappa-2x\kappa+x^2).
\]
In particular, the normalization of $\mc{C}_3$ is the disjoint union of three conics. 

\medskip
By contrast, the Weierstrass inflectionary curve $\mc{C}_2$ has geometric genus {\it one}. Indeed, $\mc{C}_2$ is an irreducible curve of arithmetic genus 3 cut out by $P_2=\frac{1}{8}(3x^4+6x^2a+ 24x- a^2)$; its projective completion to a curve in $\mb{P}^2_{x,a,z}$ (obtained by homogenizing $P_2$ with respect to $z$) has a unique singularity in $(0,1,0)$, namely the tacnode (of delta-invariant 2) with affine equation $3x^4+6x^2z-z^2=0$.

\begin{prop}\label{atomic_inflectionary_curves_over_Fq}
The values of the renormalized errors $\wt{e}_{2,p}$ are equidistributed with respect to the Sato--Tate measure on an elliptic curve with complex multiplication.
\end{prop}

\begin{figure}
\includegraphics[scale=0.6]{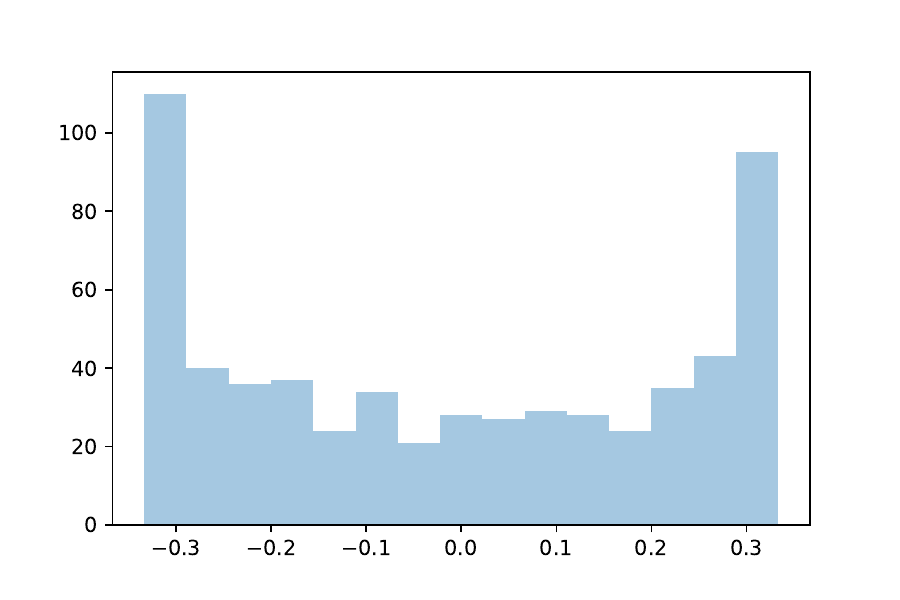}
\caption{Distribution of renormalized errors for the Weierstrass inflectionary curve $\mc{C}_2$.}
\label{sato-tate_histogram}
\end{figure}

See Figure~\ref{sato-tate_histogram} for a graphical representation of the renormalized errors $\wt{e}_{2,p}$ as $p$ varies in the range $p\leq 10000$ and $p$ splits in $\mathbb{Q}(\sqrt{-3})$. 

\begin{proof} To verify Proposition~\ref{atomic_inflectionary_curves_over_Fq}, we will show that the normalization $\wt{\mc{C}}_2$ of $\mc{C}_2$ is an elliptic curve with complex multiplication over $\mathbb{Q}(\sqrt{-3})$. Indeed, we start in $\mb{A}^2_{x,z}$, where $\mc{C}_2$ is given by the equation $\wt{P}_2:=3x^4+ 6x^2z+ 24xz^3-z^2=0$, and the tacnode is supported in the origin. Introducing an auxiliary (weighted) blow-up variable $t= \frac{x^2}{z}$, we can rewrite the pullback of $\wt{P}_2$ to $\mb{Q}[x,z,\frac{1}{z},t]$ as the system of equations $Q_1:=tz-x^2=0, Q_2:=3t^2+6t+24xz-1=0$, whose homogenization with respect to $a$ gives an affine presentation for $\wt{\mc{C}}_2$ as the intersection of quadrics in $\mb{P}^3_{x,a,z,t}$. There is a well-known recipe for converting a space cubic of the latter type to an isomorphic plane cubic; see, e.g., \cite[\S 1.4]{Pr}. Namely, identifying $Q_1$ and $Q_2$ with the $4 \times 4$ bilinear forms to which they correspond, the plane cubic presentation is $y^2= f(x)$, where $f(x):=\det(Q_1 x+ Q_2)$. In our case, we find that $f(x)=-\frac{1}{4}x^3+ 1728$. The corresponding elliptic curve $E$ has complex multiplication over $\mathbb{Q}(\sqrt{-3})$.

\medskip
Now the normalization map $\wt{\mc{C}}_2\to\mc{C}_2$ is an isomorphism over $\mb{Q}$ when restricted to the complement of the point $(0:1:0)\in\mc{C}_2$. The fiber above the latter point consists of those points $(0:1:0:t)\in\wt{\mc{C}}_2$ such that $3t^2 + 6t - 1 = 0$. It follows that for every primes $p>3$, we have $$\#\mc{C}_2(\mb{F}_p) = \#\wt{\mc{C}}_2(\mb{F}_p) - \left(\dfrac{3}{p}\right) = \#E(\mb{F}_p) - \left(\dfrac{3}{p}\right).$$
Since the Legendre symbol only takes the values $\pm 1$, it becomes negligible when renormalizing the error terms $\wt{e}_{2,p}$. Therefore the error terms of $\mc{C}_2$ obey the same distribution as those of $E$.
\end{proof}

\medskip
This is somewhat mysterious, as $\mc{C}_2$ does not itself parameterize ($x$-coordinates of) inflection points of linear series on elliptic curves\footnote{Indeed, because $n=g+2$, those atomic inflectionary curves $\mc{C}_n$ that parameterize inflection points of linear series on elliptic curves satisfy $n \geq 3$.}. It would be interesting to have a modular interpretation of $\mc{C}_2$ and its normalization.

\section*{Acknowledgements}
Many thanks are due to Jesse Kass, Matthias Wendt and Kirsten Wickelgren for explanations, helpful guidance, and patience. This project was born out of a suggestion of Kirsten's following a seminar talk by the first author in September 2018. We are grateful to the other participants of the 2019 Arizona Winter School, especially Hannah Larson and Isabel Vogt, for their inspiring ideas; and to the anonymous referees, whose comments and queries have led to an improved exposition. Finally, the first author would like to thank Indranil Biswas and Cristhian Garay L\'opez, with whom he collaborated on the problems that lay the groundwork for the current paper.

\vspace{\baselineskip}



\begin{thebibliography}{10}
\bibitem{AP} M. Aubry and M. Perret, {\it A Weil theorem for singular curves}, in ``Arithmetic, geometry, and coding theory", Pellikaan, Perret, Vl\unichar{259}du\unichar{539}, eds., de Gruyter, New York, 1996.
\bibitem{BLGHT} T. Barnet-Lamb, D. Geraghty, M. Harris, and R. Taylor, {\it A family of Calabi--Yau varieties and potential automorphy}, Publ. RIMS {\bf 47} (2011), no. 1, 29--98.
\bibitem{BZ} B. Bekker and Y. Zarhin, {\it The divisibility by 2 of rational points on elliptic curves}, St Petersburg Math. J. {\bf 29} (2018), 683--713. 
\bibitem{BKW} C. Bethea, J. Kass, and K. Wickelgren, {\it An example of wild ramification in an enriched Riemann-Hurwitz formula}, in ``Motivic homotopy theory and refined enumerative geometry", Contemp. Math. {\bf 745} (2020), 69--82.
\bibitem{BCG} I. Biswas, E. Cotterill, and C. Garay L\'opez, {\it Real inflection points of real hyperelliptic curves}, \url{arXiv:1708.08400}, Trans. AMS {\bf 372} (2019), no. 7, 4805--4827.
\bibitem{BBMMO} T. Brazelton, R. Burklund, S. McKean, M. Montoro, and M. Opie, {\it The trace of the local $\mathbb{A}^1$--degree}, Homology Homotopy Appl. {\bf 23.1} (2021), 243--255.
\bibitem{BW} T. Bachmann and K. Wickelgren, {\it $\mb{A}^1$-Euler classes: Six functors formalisms, dualities, integrality and linear subspaces of complete intersections}, J. Inst. Math. Jussieu. (2021), 1--66. \url{doi:10.1017/S147474802100027X}.
\bibitem{CG} E. Cotterill and C. Garay L\'opez, {\it Real inflection points of real linear series on an elliptic curve}, \url{arXiv:1804.06524}, Experimental Math (2019), \url{doi:10.1080/10586458.2019.1655815}.
\bibitem{CG1} E. Cotterill and C. Garay L\'opez, {\it Inflection divisors of linear series on an elliptic curve}, \url{arXiv:1903.03222}, 2018 ICM satellite conference on moduli proceedings, Mat. Contemp. {\bf 47} (2020), 73--82.
\bibitem{CLS} D. Cox, J. Little, and H. Schenck, {\it Toric varieties}, Graduate Studies in Mathematics, vol. 124, American Mathematical Society, Providence, 2011.
\bibitem{D} P. Deligne, {\it La conjecture de Weil. I.}, Inst. Hautes \'Etudes Sci. Publ. Math. {\bf 43} (1974), 273--307.
\bibitem{EH} D. Eisenbud and J. Harris, {\it Divisors on general curves and cuspidal rational curves}, Invent. Math. {\bf 74} (1983), 371--418.
\bibitem{EH2} D. Eisenbud and J. Harris, {\it Existence, decomposition, and limits of certain Weierstrass points}, Invent. Math. {\bf 87} (1987), 495--515.
\bibitem{GV} I. Gessel and G. Viennot, {\it Binomial determinants, paths, and hook length formulae}, Adv. Math. {\bf 58} (1985), no. 3, 300--321.
\bibitem{Hom} M. Homma, {\it Funny plane curves in characteristic $p>0$}, Comm. Alg. {\bf 15} (1987), no. 7, 1469--1501.
\bibitem{H} J. Huisman, {\it Algebraic moduli of real elliptic curves}, Comm. Alg. {\bf 29} (2001), no. 8, 3459--3476.
\bibitem{KRZB} E. Katz, J. Rabinoff, and D. Zureick--Brown, {\it Diophantine and tropical geometry, and uniformity of rational points on curves}, in ``Algebraic Geometry: Salt Lake City 2015", AMS Proc. Sympos. Pure Math. (2018), 231-279.
\bibitem{KW1} J. Kass and K. Wickelgren, {\it The class of Eisenbud--Khimshiashvili--Levine is the local $\mb{A}^1$-Brouwer degree}, Duke Math J. {\bf 168} (2019), no. 3, 429--469.
\bibitem{KW2} J. Kass and K. Wickelgren, {\it An arithmetic count of the lines on a smooth cubic surface}, \url{arXiv:1708.01175}, to appear in Compositio Math.
\bibitem{KW3} J. Kass and K. Wickelgren, {\it A classical proof that the algebraic homotopy class of a rational function is the residue pairing}, Linear Algebra Appl. {\bf 595} (2020), 157--181. 
\bibitem{KS} V. Kharlamov and F. Sottile, {\it Maximally inflected real rational curves}, Moscow Math. J. {\bf 3} (2003), no. 3, 947--987.
\bibitem{Kle} F. Klein, {\it Eine neue Relation zwischen den Singularit\"aten einer algebraischen Curve}, Math. Ann. {\bf 10} (1876), no. 2, 199--209.
\bibitem{Lam} T. Lam, {\it Introduction to quadratic forms over fields}, Graduate Studies in Mathematics, vol. 67, American Mathematical Society, Providence, 2005.
\bibitem{L1} M. Levine, {\it Aspects of Enumerative Geometry with Quadratic Forms}, Doc. Math. {\bf 25} (2020), 2179--2239.
\bibitem{McK} S. McKean, {\it An arithmetic enrichment of B\'ezout's theorem}, Math. Ann. {\bf 379} (2021), no. 1, 633--660. 
\bibitem{McP} W. McCallum and B. Poonen, {\it The method of Chabauty and Coleman}, in ``Explicit methods in number theory; rational points and Diophantine equations", Panoramas et synth\`eses {\bf 36} (2012), Soc. Math. de France. 
\bibitem{Mor} F. Morel, {\it $\mb{A}^1$-algebraic topology over a field}, Lec. Notes Math. {\bf 2052} (2012), Springer. 
\bibitem{P} N. Pflueger, {\it On non-primitive Weierstrass points}, Alg. Number Th. {\bf 12} (2018), 1923--1947.
\bibitem{Pr} E. Previato, {\it Poncelet's theorem in space}, Proc. Amer. Math. Soc. {\bf 127} (1999), no. 9, 2547--2556.
\bibitem{Se} J.-P. Serre, {\it Lectures on $N_X(p)$}, Research Notes in Mathematics, CRC Press, 2011.
\bibitem{SGA6} {\it S\'eminaire de G\'eom\'etrie Alg\'ebrique du Bois Marie 1966-67, Th\'eorie des intersections et th\'eor\`eme de Riemann--Roch (SGA6)}, Lec. Notes Math. {\bf 225} (1971), P. Berthelot, A. Grothendieck, and L. Illusie, eds., Springer.
\bibitem{Si} J. Silverman, {\it Some arithmetic properties of Weierstrass points: hyperelliptic curves}, Bull. Braz. Math. Soc. {\bf 21} (1990), no. 1, 11--50.
\bibitem{SW} P. Srinivasan and K. Wickelgren, {\it An arithmetic count of the lines meeting four lines in $\mb{P}^3$}, with an appendix by B. Kadets, P. Srinivasan, A.A. Swaminathan, L. Taylor, and D. Tseng, to appear in Trans. AMS.
\bibitem{SV} K.-O. St\"ohr and J. F. Voloch, {\it Weierstrass points and curves over finite fields}. Proc. London Math. Soc. {\bf 52} (1986), no. 3, 1--19.
\bibitem{Vojta} P. Vojta, {\it Jets via Hasse-Schmidt derivatives}, Diophantine geometry, CRM Series {\bf 4} (2007), 335--361.
\end{thebibliography}
\end{document}